\documentclass[a4paper,12pt]{amsart}
\usepackage{latexsym,amscd,amssymb,url}
\usepackage[all,cmtip]{xy}  
\usepackage{graphicx}
\usepackage{xcolor}
\usepackage{enumerate} 
\definecolor{dblue}{rgb}{0,0,.6}
\usepackage[colorlinks=true, linkcolor=dblue, citecolor=dblue, filecolor = dblue, menucolor = dblue, urlcolor = dblue]{hyperref}

\newtheorem{theorem}{Theorem}

\theoremstyle{plain}

\newtheorem{corollary}[theorem]{Corollary}

\newtheorem{definition}[theorem]{Definition}

\newtheorem{lemma}[theorem]{Lemma}

\newtheorem{proposition}[theorem]{Proposition}
\newtheorem{remark}[theorem]{Remark}

\newcommand{\del}{\partial}
\newcommand{\Z}{\mathbb Z}

\newcommand{\C}{\mathbb C}

\newcommand{\CP}{\mathbb P}

\newcommand{\im}{\operatorname{im}}

\newcommand{\Aut}{\operatorname{Aut}}

\newcommand{\Sym}{\operatorname{Sym}}
\newcommand{\id}{\operatorname{id}}

\newcommand{\Spec}{\operatorname{Spec}}

\newcommand{\pr}{\operatorname{pr}}

\newcommand{\codim}{\operatorname{codim}}

\newcommand{\CH}{\operatorname{CH}}
\newcommand{\supp}{\operatorname{supp}}

\newcommand{\Frac}{\operatorname{Frac}}
\newcommand{\Val}{\operatorname{Val}}

 \newcommand{\sm}{\operatorname{sm}}

\newcommand{\dashedlongrightarrow}{\xymatrix@1@=15pt{\ar@{-->}[r]&}}
\renewcommand{\longrightarrow}{\xymatrix@1@=15pt{\ar[r]&}}
\renewcommand{\mapsto}{\xymatrix@1@=15pt{\ar@{|->}[r]&}}
\renewcommand{\twoheadrightarrow}{\xymatrix@1@=15pt{\ar@{->>}[r]&}}
\newcommand{\hooklongrightarrow}{\xymatrix@1@=15pt{\ar@{^(->}[r]&}}
\newcommand{\congpf}{\xymatrix@1@=15pt{\ar[r]^-\sim&}}
\renewcommand{\cong}{\simeq}

\begin{document}  
\title[On the rationality problem for quadric bundles]{On the rationality problem for quadric bundles}

\author{Stefan Schreieder}
\address{Mathematisches Institut, LMU M\"unchen, Theresienstr.\ 39, 80333 M\"unchen, Germany.}
\email{schreieder@math.lmu.de}

\date{July 12, 2018; \copyright{\ Stefan Schreieder 2017}}
%


\begin{abstract} 
We classify all positive integers $n$ and $r$ such that (stably) non-rational complex $r$-fold quadric bundles over rational $n$-folds exist.  
We show in particular that for any $n$ and $r$, a wide class of smooth $r$-fold quadric bundles over $\CP^n_\C$ are not stably rational if 
$r\in [2^{n-1}-1,2^n-2]$.  
In our proofs we introduce a generalization of the specialization method of Voisin and Colliot-Th\'el\`ene--Pirutka which avoids universally  
$\CH_0$-trivial resolutions of singularities.  
\end{abstract}

\maketitle

\section{Introduction} 
A quadric bundle is a flat morphism of projective varieties  $f:X\longrightarrow S$, whose generic fibre is a smooth quadric; we say that such a bundle is smooth if $X$ is smooth over the ground field, which we assume algebraically closed.
We will always assume that the base $S$ is a rational variety. 
It is then an interesting and old problem, which goes back at least to the work of Artin and Mumford \cite{artin-mumford}, to decide whether $X$ is rational as well. 
By Springer's theorem \cite{springer}, $X$ is rational if $f$ admits a rational multisection of odd degree.
By a theorem of Lang \cite{lang} (cf. \cite[II.4.5]{serre}), a section exists whenever $r> 2^n-2$, where $n=\dim(S)$ and $r$ denotes the dimension of the fibres of $f$. 

Our main result is as follows.

\begin{theorem}\label{thm:classification:2}
Let $n$ and $r$ be positive integers with $r\leq 2^n-2$, and let $m\leq n$ be the unique integer with $2^{m-1}-1\leq r\leq 2^{m}-2$. 
Then there are smooth unirational complex $r$-fold quadric bundles $X$ over $S=\CP^{n-m}_{\C}\times \CP_{\C}^{m}$, such that $X$ is not stably rational.  
\end{theorem}

As aforementioned, any complex $r$-fold quadric bundle over a rational base of dimension $n$ with $r>2^n-2$ is rational 
by Lang's theorem. 
This shows that the condition on $n$ and $r$ in Theorem \ref{thm:classification:2} is optimal; that is, a smooth (stably) non-rational complex  $r$-fold quadric bundle over a rational base of dimension $n$ exists if and only if $r\leq 2^n-2$. 

For $r=1,2$ (resp.\ $r=3,4,5,6$), the first examples of (stably) irrational quadric bundles over rational bases have been produced by Artin--Mumford \cite{artin-mumford} (resp.\ Colliot-Th\'el\`ene--Ojanguren \cite{CTO}).
While those examples are singular, examples of smooth quadric bundles that are (stably) irrational have up till now only been produced for $r=1$ and $r=2$. 
Indeed, while the rationality problem is solved for many types of smooth conic bundles \cite{I,voisin,HKT,beauville4,BB,ABGP,AO}, even for smooth quadric surface bundles, progress has been made only recently by Hassett, Pirutka and Tschinkel. 
They proved that the very general fibres of three families of quadric surface bundles over $\CP_\C^2$, degenerated over plane octic curves, are not stably rational \cite{HPT,HPT2,HPT3}. 
Each family contains a dense set of smooth rational fourfolds and so they obtained the first examples of non-rational varieties with rational deformations.
The next result shows more generally that  for any positive integer $r$, deformation invariance of rationality fails for $r$-fold quadric bundles over rational bases.

\begin{theorem} \label{thm:defo}
Let $r$ be a positive integer.
Then there is a smooth complex projective family $\pi:\mathcal X\longrightarrow B$ of smooth complex varieties such that each fibre $X_b=\pi^{-1}(b)$ is an $r$-fold quadric bundle over some complex projective space, satisfying the following:
\begin{enumerate}
\item for very general $t\in B$, the fibre $X_t$ is not stably rational;
\item the set $\{ b\in B\mid X_b\ \text{is rational}\}$ is dense in $B$ for the analytic topology.
\end{enumerate}
\end{theorem} 
 

More explicitly, we  discuss now the rationality problem for a natural and interesting class of $r$-fold quadric bundles over $\CP_\C^n$. 
We start with a generically non-degenerate line bundle valued quadratic form $q:\mathcal E\longrightarrow \mathcal O_{\CP^n_\C}(l)$, where $\mathcal E:=\bigoplus_{i=0}^{r+1}\mathcal O_{\CP^n_\C}(-l_i)$ is a split vector bundle on $\CP^n_\C$.
If $q_s\neq 0$ for all $s\in \CP^n_\C$, then $X:=\{q=0\}\subset \CP(\mathcal E)$ is an $r$-fold quadric bundle over $\CP^n_\C$.
We may identify $q$ with a symmetric matrix $A=(a_{ij})$ of homogeneous polynomials of degrees $|a_{ij}|=l_i+l_j+l$. 
Locally over $\CP^n_\C$, $X$ is given by 
$$
\sum_{i,j=0}^{r+1} a_{ij}z_iz_j =0 .
$$
The deformation type of $X$ depends only on the integers $d_i:=2l_i+l$, which have all the same parity,  
cf.\ Section \ref{subsec:quadricbundles} below.
We call any such bundle of type $(d_i)_{0\leq i\leq r+1}$.
We then have the following; see Theorem \ref{thm:higherquadrics:2} and Remark \ref{rem:defotype} below for a more general statement.  

\begin{theorem} \label{thm:type}
Let $n,r$ be positive integers with 
$2^{n-1}-1\leq r \leq 2^{n}-2$, and let $d_0,\dots ,d_{r+1}$  be integers of the same parity such that $d_i\geq 2^n+n-1$ for all $i$.
Then a very general complex $r$-fold quadric bundle of type $(d_i)_{0\leq i\leq r+1}$ over $\CP^n_\C$ is not stably rational.
\end{theorem}

The lower bound $2^n+n-1$ on the degrees is bounded from above by $n+2r+1$.
As an example, we thus see that for $n,r$ as above, very general complex hypersurfaces $X\subset \CP^n_\C\times \CP^{r+1}_\C$ of bidegree $(d,2)$ with $d\geq n+2r+1$ are not stably rational.  
In contrast,  if $r\geq 2$, some smooth hypersurfaces of that kind are rational; in fact, for $r\geq 2$, all examples in Theorem \ref{thm:type} have rational deformations, see Corollary \ref{cor:rational-defotype} below. 

As another application, we consider singular hypersurfaces $X\subset \CP^{N+1}_\C$ of degree $d$.
If $X$ is not a cone and contains a singular point whose multiplicity is roughly as large as the degree, then $X$ tends to be quite close to being rational, no matter how large $d$ is.  
For instance, a single point $x\in X$ of multiplicity $d-1$ forces $X$ to be rational.
In contrast, Theorem \ref{thm:type} implies 
that many degree $d$ hypersurfaces with points of multiplicity $d-2$ are not even stably rational.

\begin{corollary} \label{cor:hyper}
Let $n$ and $r$ be positive integers with $2^{n-1}-1\leq r\leq  2^n-2$.
Set $N:=n+r$ and $m:=2^n+n+1$. 
Then a very general complex hypersurface $X\subset \CP^{N+1}_\C$ of degree $d\geq m$ and with multiplicity $d-2$ along an $r$-plane is not stably rational.
\end{corollary}

The above hypersurfaces are birational to $r$-fold quadric bundles over $\CP^n_\C$, cf.\ Lemma \ref{lem:singhyper}.
The upper bound $r\leq 2^n-2$ is thus sharp by the aforementioned result of Lang \cite{lang}. 
The lower bound $m$ on the degree satisfies
$m \in [N+3,2N-n+3]$ and it lies on the boundary of that interval if $r=2^{n}-2$ or $r=2^{n-1}-1$, respectively.

Building on work of Koll\'ar \cite{kollar}, Totaro showed \cite{totaro} that a very general  smooth complex hypersurface $X\subset \CP^{N+1}_\C$ of degree $d\geq 2\lceil (N+2)/3 \rceil $ is not stably rational.  
Our bounds differ roughly by a factor $\lambda \in [\frac{3}{2},3]$.
While smooth hypersurfaces of such large degree are not even uniruled, our singular hypersurfaces are rationally connected.

The proofs of the above results are based on two main ingredients, which we explain in the following two subsections respectively.

\subsection{Examples \`a la Artin--Mumford and Colliot-Th\'el\`ene--Ojanguren in higher dimensions} \label{subsec:singbundles}
For any complex projective variety $Y$, there are unramified cohomology groups $H^ i_{nr}(\C(Y)\slash \C,\Z/l)$, which are stable birational invariants of $Y$.
These invariants have been introduced by Colliot-Th\'el\`ene and Ojanguren \cite{CTO} in their reinterpretation of the celebrated Artin--Mumford example \cite{artin-mumford}.
The results in \cite{artin-mumford} and \cite{CTO} show (cf.\ Lemma \ref{lem:Pfister} below) that for $n=2$ and $r=1,2$, or $n=3$ and $r=3,4,5 ,6$, 
there is a singular unirational $r$-fold quadric bundle $Y$ over $\CP^n_\C$ with $H^n_{nr}(\C(Y)\slash \C,\Z\slash 2)\neq 0$. 
For $n=r=2$, different examples with the same property have recently been constructed by Pirutka \cite{Pirutka} 
and Hassett--Pirutka--Tschinkel \cite{HPT}.

Using an algebraic approach of Peyre \cite{peyre}, 
Asok showed that for arbitrary positive integers $n$ and $r$ with $2^{n-1}-1\leq r\leq 2^{n}-2$, 
there is a collection of singular unirational $r$-fold quadric bundles $Y_1,\dots ,Y_s$ over $\CP^{2n}_\C$, with $s=\binom{2n}{n}-1$, such that their common fibre product over $\CP^{2n}_\C$ has nontrivial unramified cohomology in degree $n$, see \cite[Theorem 4.2]{asok} and Lemma \ref{lem:Pfister} below.  

For $r\geq 7$, $r$-fold quadric bundles over rational bases with nontrivial unramified cohomology are not known.
Generalizing \cite{artin-mumford} and \cite{CTO}, the next result provides such examples for any $r$.

\begin{theorem}\label{thm:H_nr}
Let $n$ and $r$ be positive integers with $2^{n-1}-1\leq r\leq 2^{n}-2$.
Then there is a unirational complex projective $r$-fold quadric bundle $Y\longrightarrow \CP^n_\C$ with 
$$
H^n_{nr}(\C(Y)\slash \C,\Z\slash 2)\neq 0 .
$$ 
\end{theorem}

For a slightly more general result, see Theorem \ref{thm:H_nr:2} below.
As in \cite{asok}, 
the above result relies 
on Voevodsky's proof of the Milnor conjecture \cite{voevodsky}.   

\subsection{A specialization method without resolutions} \label{subsec:intro:specialize}
  
Voisin \cite{voisin} introduced and Colliot-Th\'el\`ene--Pirutka \cite{CT-Pirutka} developed further a specialization technique which led to numerous applications in the study of (stable) rationality properties of rationally connected varieties, see for instance \cite{AO,ABGP,beauville4,BB,CT-Pirutka2,HKT,HPT,HPT2,HPT3,HT,KO,okada,peyre2,totaro}.
Roughly speaking, in order to prove stable non-rationality of a projective variety $X$, one has to find a degeneration $Y$ of $X$ which admits both, some obstruction for stable rationality (e.g.\ nontrivial unramified cohomology) and a universally $\CH_0$-trivial resolution of singularities $\tau:\widetilde Y\longrightarrow Y$.
In order to check this last property in practice, one has to provide explicit local charts for the resolution $\widetilde Y$ and show that all scheme-theoretic fibres of $\tau$ have universally trivial Chow groups of zero-cycles. 
This is a quite subtle condition, whose verification was one of the main technical difficulties in several applications mentioned above, see for instance \cite{CT-Pirutka2,HPT,HPT2,KO}.
In particular, the method applies only to situations where $Y$ has very mild singularities and resolutions can be described explicitly.  
For instance, it had been impossible to apply the method to several (reasonable) special fibres $Y$, where obstructions for rationality were known, but the singularities did not seem to allow manageable resolutions.

The main idea of this paper is to replace the existence of a universally $\CH_0$-trivial resolution of $Y$ by a weaker condition, which is easier to check, see Proposition \ref{prop:degeneration} below. 
This leads to more general specialization theorems which also apply in situations where it seems impossible to compute a resolution of singularities explicitly, let alone to check that a universally $\CH_0$-trivial one exists. 
 
To state such a result, note that we define in this paper CTO type quadrics over rational function fields and produce examples in arbitrary dimensions, see Definition \ref{def:CTOtype} and Proposition \ref{prop:examplesofCTO-type}. 
These quadrics appear as generic fibres in the examples of Theorem \ref{thm:H_nr}. 
We then prove the following specialization theorem; for what it exactly means that a variety degenerates or specializes to another variety, see Section \ref{subsec:conventions:specialize} below.

\begin{theorem} \label{thm:CTO} 
Let $X$ be a projective variety which specializes to a complex projective variety $Y$ with a morphism $f:Y\longrightarrow S$ to a rational complex $n$-fold $S$ with $n\geq 2$.
If the generic fibre $Y_{\eta}$ of $f$ is smooth and stably birational to a CTO type quadric $Q$ over $\C(S)$, then  
$X$ is not stably rational.
\end{theorem}

Remarkably, the only condition on $Y$ that we have to impose in the above theorem concerns the generic fibre of $f:Y\longrightarrow S$.
For instance, $f$ does not need to be flat and $Y$ does not need to have a universally $\CH_0$-trivial resolution.
In fact, there is no assumption whatsoever on the singularities of $Y$ at points which do not dominate $S$. 

This significantly extends the number of possible applications. 
The main point is that for any smooth quadric $Q$ over $\C(\CP^n)$, and for any rational $n$-fold $S$, there is a wide range of different models $f:Y\longrightarrow S$ with $Q$ as generic fibre.
If $Q$ is of CTO type, then the above theorem applies and so any variety which specializes to $Y$ is not stably rational.

We emphasize that in general one must be quite careful when trying to deduce non-rationality for $X$ from non-rationality properties of some specialization $Y$ of $X$. 
For instance, a smooth cubic surface is rational and degenerates to a cone over an elliptic curve, which is non-rational. 
In fact, the situation is worse: $\CP^N$ specializes to the cone over any hypersurface $Z\subset \CP^N$ (cf.\ \cite[\S 4]{totaro}) and hence to a projective variety $Y$, stably birational to any given projective variety of dimension $N-1$.
For instance, if $N\geq 4$, we may choose a specialization $Y$ of $\CP^N_\C$ with a rational map $f:Y\dashrightarrow \CP^n_\C$ whose generic fibre is stably birational to a CTO type quadric.  
In Theorem \ref{thm:CTO}, such degenerations are excluded by the assumption that $f$ is a morphism with smooth generic fibre.
It is however possible to weaken those assumptions so that $f:Y\dashrightarrow S$ is only a dominant rational map, but $Y$ must have sufficiently mild singularities locally along the closure of a general fibre of $f$, see Theorem \ref{thm:CTO:2} for the precise statement.

\begin{remark}
The quadric surfaces over $\C(\CP^2)$, recently constructed by Pirutka \cite{Pirutka} and Hassett, Pirutka and Tschinkel \cite{HPT}, are not of CTO type. 
Nonetheless, our specialization method without resolutions works also for those quadrics.  
This simplifies \cite{HPT,HPT2,HPT3}, but it also yields much more general results which seemed inaccessible before. 
The details appeared elsewhere \cite{Sch2}. 
\end{remark}

\section{Preliminaries} 
\subsection{Conventions and notations} 
All schemes are separated.
A variety is an integral scheme of finite type over a field. 
Two varieties $X$ and $Y$ over a field $k$ are stably birational, if $X\times \CP^m_k$ is birational (over $k$) to $Y\times \CP^n_k$ for some $n,m\geq 0$.
A resolution of a variety $Y$ is a proper birational morphism of varieties $\tau:\widetilde Y\longrightarrow Y$, with $\widetilde Y$ smooth.
If $Z\subset Y$ is a closed subscheme of a variety $Y$, then a log resolution of the pair $(Y,Z)$ is a resolution of singularities $\tau:\widetilde Y\longrightarrow Y$ such that the reduced subscheme which underlies $\tau^{-1} Z$ is a simple normal crossing divisor. 
A property is said to hold for a very general point of a scheme if it holds at all closed points outside some countable union of proper closed subsets. 

\subsection{What it means that a variety specializes or degenerates to another one}\label{subsec:conventions:specialize}
We say that a variety $X$ over a field $L$ specializes (or degenerates) to a variety $Y$ over a field $k$, with $k$ algebraically closed, if there is a discrete valuation ring $R$ with residue field $k$ and fraction field $F$ with an injection $F\hookrightarrow L$ of fields, together with a flat proper morphism $\mathcal X\longrightarrow \Spec R$ of finite type, such that $Y$ is isomorphic to the special fibre $Y\cong \mathcal X\times k$ and  $X\cong \mathcal X\times L$ is isomorphic to a base change of the generic fibre $\mathcal X\times F$.

The next lemma shows that this terminology allows quite some flexibility.

\begin{lemma}\label{lem:specialize}
Let $\pi:\mathcal X\longrightarrow B$ be a flat proper morphism of complex varieties with integral fibres, and let $0\in B$ be a closed point.
Then for any very general point $t\in B$, the fibre $X_t$ specializes to $X_0$.
\end{lemma}
\begin{proof}
The family $\pi$ is obtained as base change of some 
family $\pi':\mathcal X'\longrightarrow B'$ defined over some countable algebraically closed subfield $k\subset \C$. 
Let  $U\subset B(\C)$ be the union of all closed points $b\in B$, which do not lie on $Z'\times_k\C$ for some proper subvariety $Z'\subsetneq B'$.  
Since there are only countably many such subvarieties $Z'$, any very general point of $B$ lies in $U$.
Moreover, for any $t\in U$, there is a field isomorphism $\varphi:\overline {\C(B)}\stackrel{\sim}\longrightarrow \C$ which identifies the geometric generic fibre $\mathcal X\times \overline{\C(B)}$ with the very general fibre $X_t$, see for instance \cite[Lemma 2.1]{vial}. 
This shows that the fibres $X_t$ with $t\in U$ are all abstractly isomorphic (i.e.\ differ only by the action of  $\Aut(\C)$)  and so it suffices to find one $t\in U$ such that $X_t$ degenerates to $X_0$.
Hence, we may reduce to the case where $B$ is a curve.  
Taking normalizations, we may also assume that $B$ is smooth.
Using again that the geometric generic fibre of $\pi$ is abstractly isomorphic to any very general fibre, the statement is now clear because $\mathcal O_{B,0}$ is a discrete valuation ring if $B$ is a smooth curve. 
\end{proof}

\subsection{Chow groups of zero-cycles} \label{subsec:CH0}
A morphism $f:X\longrightarrow Y$ of varieties over a field $k$ is universally $\CH_0$-trivial,  if $f_\ast :\CH_0(X\times L)\stackrel{\cong}\longrightarrow \CH_0(Y\times L)$ is an isomorphism for all field extensions $L$ of $k$.
If the structure morphism $f:X\longrightarrow \Spec k$ is universally $\CH_0$-trivial, then we say that the Chow group of zero-cycles of $X$ is universally trivial.
If $X$ is smooth and proper, this is equivalent to the existence of an integral decomposition of the diagonal $\Delta_X\in \CH_{\dim(X)}(X\times X)$ as in (\ref{eq:DeltaX}) below. 
The Chow group of zero-cycles of a smooth projective variety $X$ over a field is a stable birational invariant, see \cite[Lemme 1.5]{CT-Pirutka} and \cite[Theorem 1.1]{totaro} and references therein.

\subsection{Galois cohomology of fields}
Let $K$ be a field of characteristic coprime to $l$.
We identify the Galois cohomology group $H^n(K,\mu_l^{\otimes n})$  with the \'etale cohomology group $H^n_{\text{\'et}}(\Spec(K),\mu_l^{\otimes n})$, where $\mu_l\subset \mathbb G_m$ denotes the group of $l$-th roots of unity. 
We also use the identification $H^1(K,\mu_l)\cong K^\ast\slash(K^\ast)^l$, induced by the Kummer sequence.
For $a_1,\dots ,a_n\in K^\ast$, we denote by $(a_1,\dots ,a_n)\in H^n(K,\mu_l^{\otimes n})$ the class obtained by cup product. 
Classes of this form are called symbols.

If $A$ is a discrete valuation ring with fraction field $K$ and residue field $\kappa$ whose characteristic is coprime to $l$, then there are residue maps $\del^n_A:H^n(K,\mu_l^{\otimes n})\longrightarrow H^{n-1}(\kappa,\mu_l^{\otimes (n-1)})$.
If $\nu$ denotes the corresponding valuation on $K$, we also write $\del^n_\nu=\del^n_A$. 

The following lemma computes the residue of a symbol explicitly in the case of $\mu_2$-coefficients, where squares can be ignored. 

\begin{lemma}\label{lem:residue}
Let $A$ be a discrete valuation ring with residue field $\kappa$ and fraction field $K$, both of characteristic different from $2$.  
Suppose that $-1$ is a square in $K$.
Let $\pi\in A$ be a uniformizer, $0\leq m\leq n$ be integers and let $a_1\dots ,a_{n}\in A^\ast$ be units in $A$.
Then the following identity holds in $H^{n-1}(\kappa, \mu_2^{\otimes (n-1)})$:
$$
\del^n_A(\pi a_1,\dots ,\pi a_m ,a_{m+1},\dots ,a_n)= \left( \sum_{i=1}^m (\overline a_1,\dots ,\widehat {\overline a_i},\dots ,\overline a_m) \right) \cup (\overline a_{m+1},\dots ,\overline a_n)  ,
$$
where $\overline a_i\in \kappa$ denotes the image of $a_i$ in $\kappa$ and $(\overline a_1,\dots ,\widehat {\overline a_i},\dots ,\overline a_m)$ denotes the symbol where $\overline a_i$ is omitted.
Here we use the convention that the above sum $\sum_{i=1}^m$ is one if $m=1$ and it is zero if $m=0$. 
\end{lemma}

\begin{proof}
The cases $m=0,1$ follow for instance from \cite[Proposition 1.3]{CTO}.
In order to prove the lemma, it thus suffices to show the following: 
$$
(\pi a_1,\dots ,\pi a_m ,a_{m+1},\dots ,a_n)=\left( \sum_{i=0}^m (a_1,\dots ,a_{i-1},\pi,a_{i+1},\dots ,a_m) \right) \cup (a_{m+1},\dots ,a_n) ,
$$ 
where the summand for $i=0$ is understood to be $(a_1,\dots ,a_m)$.
To prove this identity, recall that the Steinberg relations $(a,1-a)=0$ for $a\in K\setminus\{0,1\}$ imply $(a,-a)=0$ for all $a\in K^\ast$, see for instance \cite[Lemma 2.2]{kerz}. 
Since $-1$ is a square in $K$, $(\pi,\pi)=0$.
Using this, the formula follows immediately.  
\end{proof}

We will use the following compatibility of residues, see \cite[p.\ 143]{CTO}.

\begin{lemma}\label{lem:residue:compatible}
Let $f:\Spec B\longrightarrow \Spec A$ be a surjective morphism of schemes, where $A$ and $B$ are discrete valuation rings with fraction fields $K=\Frac A$ and $L=\Frac B$ and residue fields $\kappa_A$ and $\kappa_B$ of characteristic different from $2$, respectively.
Then there is a commutative diagram
\begin{align*} 
\xymatrix{
H^n(L,\mu_2^{\otimes n}) \ar[r]^-{\del_B^n} & H^{n-1}(\kappa_B,\mu_2^{\otimes (n-1)}) \\
H^n(K,\mu_2^{\otimes n}) \ar[r]^-{\del_{A}^n}  \ar[u]^{f^\ast}&H^{n-1}(\kappa_A,\mu_2^{\otimes (n-1)})  \ar[u]_{e\cdot f^\ast} ,
}
\end{align*}
where $e=\nu_B(\pi_A)\in \Z$ is the valuation with respect to $B$ of a uniformizer $\pi_A$ of $A$.
\end{lemma}

Finally, we will use the following basic vanishing result, see \cite[II.4.2]{serre}. 

\begin{theorem} \label{thm:milne}
Let $K$ be the function field of an $n$-dimensional variety over an algebraically closed field of characteristic different from $2$.
Then, $H^i(K,\mu_2^{\otimes i})=0$ for all $i>n$.
\end{theorem}

\subsection{Rost cycle modules} \label{subsec:Rostcycle}
Let $k$ be a field.
For any finitely generated field extension $L$ of $k$, we denote by $ \Val(L/k)$ the set of all geometric discrete valuations of rank one on $L$ over $k$. 
Such valuations are characterized by the property that the corresponding valuation ring $\mathcal O_\nu\subset L$ is the local ring $\mathcal O_{X,x}$ at a codimension one point $x\in X^{(1)}$ of some normal variety $X$ over $k$ with $k(X)=L$, see \cite[Proposition 1.7]{merkurjev}.

A Rost cycle module  $M^\ast$  over $k$ is a functor from the category of finitely generated field extensions of $k$ to $\Z$-graded abelian groups with some additional properties, see \cite{rost} and \cite[Section 2]{merkurjev}. 
An important one for us is the existence of residue maps $\del_\nu^i:M^i(L)\longrightarrow M^{i-1}(E)$, for all  $\nu\in \Val(L\slash k)$, where $L\slash k$ is a finitely generated field extension and $E$ is the residue field of $\nu$.
The group of unramified elements is 
$$
M^i_{nr}(L):=\{\alpha \in M^i(L) \mid  \del^i_\nu\alpha=0\ \text{for all $\nu\in \Val(L/k)$} 
\}.
$$ 
A class $\alpha\in M^i_{nr}(L)$ is called nontrivial, if it is not in the image of $M^i(k)\longrightarrow M^i_{nr}(L)$.

If $X$ is a variety over $k$, then we write $M^i_{nr}(X):=M^i_{nr}(k(X))$.
If $X$ and $Y$ are smooth proper varieties over $k$, then for any cycle $\Gamma\in \CH_{\dim(X)}(X\times Y)$, there is a homomorphism
$$
\Gamma^\ast:M^i_{nr}(Y)\longrightarrow M^i_{nr}(X) ,
$$
which is trivial whenever $\Gamma$ does not dominate $X$, see \cite[RC-I and proof of RC.9]{karpenko-merkurjev}. 
Via these actions, unramified cohomology descends to a functor on the category of integral correspondences between smooth and proper $k$-varieties, see \cite[RC.3-4]{karpenko-merkurjev}.
If $\Gamma$ is the graph of a rational map $f:X\dashrightarrow Y$, we obtain pullback maps $\Gamma^\ast=f^\ast$. 

\subsection{Unramified cohomology} \label{subsec:unramified}
An important example of a Rost cycle module over a field $k$ is given by Galois cohomology $M^i(L)=H^i(L,\mu_l^{\otimes i})$, 
with $l$ coprime to $\operatorname{char}(k)$. 
The corresponding unramified cohomology groups are denoted by $H^i_{nr}(L,\mu_l^{\otimes i})$; if we want to emphasize the base field $k$, we also write $H^i_{nr}(L\slash k,\mu_l^{\otimes i})$ for this group.
If $k$ is algebraically closed and $i\geq 1$, then $H^i(k,\mu_l^{\otimes i})=0$  and so any $0\neq \alpha\in H^i_{nr}(L\slash k,\mu_l^{\otimes i})$ is a nontrivial unramified cohomology class in the sense of Section \ref{subsec:Rostcycle} above.
Originally, unramified cohomology has been defined by Colliot-Th\'el\`ene--Ojanguren \cite{CTO} as the subgroup of all elements $\alpha\in H^i(L,\mu_l^{\otimes i})$ that have trivial residue at all discrete valuations of rank one on $L$ over $k$ (and not only at the geometric ones).
It follows from \cite[Theorem 4.1.1]{CT} that the two definitions coincide if resolutions of singularities exist over $k$ (e.g.\ if $k=\C$).


If $X$ is a variety over $\C$,  $H_{nr}^i(\C(X)\slash \C,\mu_l^{\otimes i})$ is a stable birational invariant of $X$, see \cite[Proposition 1.2]{CTO}.
If additionally $X$ is smooth and projective, then $H^3_{nr}(\C(X)\slash \C,\mu_l^{\otimes 3})$ and $H^4_{nr}(\C(X)\slash \C,\mu_l^{\otimes 4})$ are related to failure of the integral Hodge conjecture for codimension two cycles on $X$ and to torsion in the third Griffiths group, annihilated by the Abel--Jacobi map, respectively,  see \cite{CT-voisin} and \cite{voisin-Hnr4}.

\section{Quadric bundles and quadrics over non-closed fields} \label{sec:quadrics}

\subsection{Quadratic forms and Pfister neighbours} \label{subsec:quadratic-forms}
Let $K$ be a field of characteristic different from $2$.
Any quadratic form $q$ on an $n$-dimensional $K$-vector space can be diagonalized, $q=\langle a_1,\dots ,a_n\rangle$ for some $a_i\in K$, and we call $n$ the dimension of $q$. 
We associate to $q$ the quadric hypersurface $Q:=\{q=0\}\subset \CP^{n-1}_K$, given by $\sum_ia_i z_i^2=0$.
Two quadratic forms are similar if and only if the corresponding quadric hypersurfaces are isomorphic.
The form $q$ is isotropic if and only if $Q$ admits a $K$-rational point.

The form $q$ is called an ($n$-fold) Pfister form, if it is isomorphic to the tensor product of forms of type $\langle 1,-a_i\rangle$ with nonzero $a_i\in K$, where $i=1,\dots ,n$.
We denote this tensor product by $\langle\langle a_1,\dots,a_n\rangle\rangle$; it is a form of dimension $2^n$.
The sign can be ignored if $-1$ is a square in $K$. 
A non-degenerate quadratic form $q_1$ is called a Pfister neighbour if it is similar to a subform of a Pfister form $q_2$ with $2\dim(q_1)>\dim(q_2)$.

\subsection{Birational geometry of quadrics}
Let $K$ be a field of characteristic different from $2$.
We say that two quadratic forms $q_1$ and $q_2$ over $K$ are stably birational, if the associated quadric hypersurfaces are stably birational over $K$.
The following lemma is well-known (cf.\ \cite[Proposition 2]{hoffmann}); for more results on the birational geometry of quadrics, we refer to \cite{totaro2} and references therein.

\begin{lemma} \label{lem:Pfister}
Let $q_2$ be a Pfister form over $K$. 
Then any Pfister neighbour $q_1$ of $q_2$ is stably birational to $q_2$. 
\end{lemma}
\begin{proof} 
Let $Q_i$ be the quadric associated to $q_i$.
It suffices to prove that the generic fibre of $\pr_i:Q_1\times Q_2\longrightarrow Q_i$ is rational for $i=1,2$.  
Since $q_1$ is a subform of $q_2$, $Q_2$ has a $K(Q_1)$-rational point and so this is clear for $i=1$.
Conversely, $q_2$ is isotropic over $K(Q_2)$ and so $Q_1$ has a $K(Q_2)$-rational point, because  $2\dim(q_1)>\dim(q_2)$ and isotropic Pfister forms
are hyperbolic \cite[II.9.10]{EKM}. 
This proves the lemma.
\end{proof}

\begin{remark} \label{rem:4diml-quadrics}
By a result of Hoffmann \cite[Proposition 2]{hoffmann}, an anisotropic quadratic form $q_1$ over $K$ is stably birational to an anisotropic Pfister form $q_2$ if and only if $q_1$ is  a Pfister neighbour of $q_2$.
\end{remark}

The following unirationality criterion goes back to Colliot-Th\'el\`ene and Ojanguren. 

\begin{lemma} \label{lem:unirational}
Let $n\geq 2$, and let $K=\C(x_1,\dots ,x_n)$ be the function field of $\CP^n_\C$. 
Consider the quadratic form $q=\langle 1,a_1,a_2,\dots ,a_r\rangle$ over $K$ for some $a_i\in K^\ast$. 
Suppose that $a_1=f/g$ with $f,g\in \C[x_1,\dots ,x_n]$, 
satisfying one of the following:
\begin{enumerate}
\item $f$ and $g$ are of degree at most one;
\item $f$ and $g$ have degree at most two and the homogenization $q\in L[x_0,\dots ,x_n]$ of  $gz^2-f$, where $L=\C(z)$, is a quadratic form of rank $\geq 3$ over $L$. 
\end{enumerate} 
Then the quadric hypersurface $Q$ determined by $q$ is unirational over $\C$; more precisely, a degree two extension of $K(Q)$ is purely transcendental over $\C$. 
\end{lemma}

\begin{proof}
The proof is similar to the arguments in \cite[Propositions 2.1 and 3.1]{CTO}.  
If $a_1$ is a square, then $Q$ is rational over $K$ and so the statement is clear.
Otherwise, $K':=K[z]\slash(z^2-a_1)$ is a field. 
Since $Q\times K'$ has a $K'$-rational point, it is rational over $K'$.
It thus suffices to see that $K'\cong \C(\CP^n)$. 
To this end, consider $L=\C(z)$ and let $Z\subset \CP^{n}_{L}$ be the projective closure of $\{gz^2-f=0\}$.
By construction, $K'=L(Z)$ and so it suffices to prove that $Z$ is rational over $L$.
This is clear if $f$ and $g$ are linear.
Otherwise, our assumptions imply that $Z$ is a cone over a smooth quadric $Z'$ over $L$ of dimension at least one.  
Since $L=\C(z)$ is a $C_1$-field, $Z'$ has a $L$-rational point and so $Z$ is rational.  
This concludes the lemma. 
\end{proof}

\subsection{A result of Orlov, Vishik and Voevodsky} 
Voevodsky's proof of the Milnor conjecture \cite{voevodsky} together with an exact sequence of Orlov, Vishik and Voevodsky \cite[Theorem 2.1]{OVV}, implies the following important result. 

\begin{theorem}[Orlov--Vishik--Voevodsky] \label{thm:OVV}
Let $K$ be a field of characteristic zero, and let $q$ be a Pfister neighbour of the Pfister form $\langle\langle a_1,\dots ,a_n \rangle\rangle$, with $a_i\in K^\ast$.
Let $f:Q\longrightarrow \Spec K$ be the projective quadric associated to $q$.
Then the kernel of
$$
f^\ast:H^n(K,\mu_2^{\otimes n})\longrightarrow H^n(K(Q),\mu_2^{\otimes n})
$$ 
is generated by $(a_1,\dots ,a_n)$.
\end{theorem}
\begin{proof}
By \cite[Theorem 2.1]{OVV} and  
 \cite{voevodsky}, the result holds for the Pfister neighbour $q=\langle\langle a_1,\dots ,a_{n-1}\rangle\rangle \oplus \langle -a_n \rangle$.
The stated result follows therefore from Lemma \ref{lem:Pfister}, because $\im(f^\ast)\subset H^n_{nr}(K(Q)\slash K,\mu_2^{\otimes n})$ and unramified cohomology is a stable birational invariant \cite[Proposition 1.2]{CTO}.
\end{proof}

\subsection{Quadrics \`a la Artin--Mumford and Colliot-Th\'el\`ene--Ojanguren}
\label{subsec:CTO}
 
The following definition summarizes the conditions in  \cite[Propositions 2.1 and 3.1]{CTO} of Colliot-Th\'el\`ene and Ojanguren's paper, where the cases $n=2$ and $3$ are studied.

\begin{definition} \label{def:CTOtype}
Let $n\geq 2 $ be an integer and consider the function field $K=\C(\CP^{n})$.
Suppose that there are elements $a_1,\dots ,a_{n-1},b_1,b_2\in K^\ast$ such that for $j=1,2$, the class $\alpha_j:=(a_1,\dots ,a_{n-1},b_j)\in H^n(K,\mu_2^{\otimes n})$ is nonzero 
and satisfies the following:
\begin{enumerate} 
\item[($\ast$)]
for any $\nu\in \Val(K\slash \C)$, $\del_\nu^n \alpha_j=0$ for $j=1$ or $2$. 
\end{enumerate}
Then any projective quadric $Q=\{q=0\}$ over $K$ defined by a Pfister neighbour $q$ of the $n$-fold Pfister form $\langle\langle a_1,\dots ,a_{n-1},b_1b_2 \rangle\rangle$ is called a quadric of CTO type. 
\end{definition}
 
Since Pfister neighbours are non-degenerate by definition, we note that CTO type quadrics are always smooth.
 
 The results in \cite{CTO} can be summarized as follows: if $n=2$ or $3$, then CTO type quadrics exist and have nontrivial unramified $\Z\slash 2$-cohomology in degree $n$; 
the Artin--Mumford example  \cite{artin-mumford} is a CTO type conic over $\C(\CP^2)$. 

While the proof that CTO type quadrics exist for $n=2,3$ in \cite{CTO} is quite subtle, the argument which proves non-triviality of $H^n_{nr}(K(Q)\slash \C,\mu_2^{\otimes n})$ works (thanks to Theorem \ref{thm:OVV} of Orlov--Vishik--Voevodsky) in arbitrary dimensions as follows. 

\begin{proposition}[Colliot-Th\'el\`ene--Ojanguren] \label{prop:CTO}
Let $n\geq 2$ and let $f:Q\longrightarrow \Spec K$ be a CTO type quadric over $K=\C(\CP^{n})$.
Then, $0\neq  f^\ast \alpha_1\in H^n_{nr}(K(Q)\slash \C,\mu_2^{\otimes n})$.
\end{proposition}
\begin{proof}
By Theorem \ref{thm:OVV}, $f^\ast \alpha_1=f^\ast \alpha_2$ and we denote this class by $\alpha'$.
Let $\nu\in \Val(K(Q)\slash \C)$ and consider the restriction $\mu:=\nu|_K$.
If $\mu$ is trivial, then $\del_\nu \alpha'=0 $ by Lemma \ref{lem:residue}.
Otherwise, $\mu \in \Val(K\slash \C)$ by \cite[Proposition 1.4]{merkurjev}.
By Lemma \ref{lem:residue:compatible}, there is some $e\in \Z$ such that $\del^n_\nu \alpha'=e\cdot f^\ast (\del^n_{\mu}\alpha_j)$ for $j=1,2$.
Hence, $\del^n_\nu \alpha'=0$, because $\del_{\mu}^n \alpha_j=0$ for $j=1$ or $2$ by assumptions.
Therefore, $\alpha'=f^\ast \alpha_1\in H^n_{nr}(K(Q)\slash \C,\mu_2^{\otimes n})$ is unramified over $\C$. 
To prove that it is nonzero, it suffices by Theorem \ref{thm:OVV} to see that $\alpha_1\neq 0$ and $\alpha_1\neq \alpha_1+\alpha_2$.
This follows from $\alpha_j\neq 0$ for all $j=1,2$. 
\end{proof}

\begin{remark} 
Proposition \ref{prop:CTO} implies that CTO type quadrics are always anisotropic.
\end{remark}

\subsection{Quadric bundles} \label{subsec:quadricbundles}
In this section we work over an algebraically closed field $k$ of characteristic different from two; 
as we will be applying Bertini's theorem on base point free linear series, we will sometimes need to restrict further to the case $\operatorname{char}(k)=0$.
%
A quadric bundle is a flat morphism $f:X\longrightarrow S$ of projective varieties over $k$ whose generic fibre is a smooth quadric over $k(S)$. 
If $f$ is not assumed to be flat, then $X$ is called weak quadric bundle.

Let $q:\mathcal E\longrightarrow L$ 
be a generically non-degenerate line bundle valued quadratic form on some vector bundle $\mathcal E$ on $S$ such that $q_s\neq 0$ for all $s\in S$.
Then $q\in H^0(S,\Sym^2(\mathcal E^\vee)\otimes L)$ and the hypersurface $X:=\{q=0\}\subset \CP(\mathcal E)$ is a quadric bundle over $S$; flatness follows because all fibres $X_s=\{q_s=0\}\subset \CP(\mathcal E_s)$ have the same Hilbert polynomial. 
The degeneration locus on $S$ is given by the divisor where $q$ does not have full rank.

We will always assume that $\mathcal E=\bigoplus_{i=0}^{r+1} L_i^{-1}$ splits into a sum of line bundles.
Under this assumption, $q$ corresponds to a symmetric matrix $A=(a_{ij})$, where $a_{ij}$ is a global section of $L_i\otimes L_j\otimes L $. 
Locally over the base $S$, $X$ is given by
\begin{align} \label{eq:X:2}
\sum_{i,j=0}^{r+1} a_{ij} z_iz_j=0  ,
\end{align}
where $z_i$ denotes a local coordinate which trivializes $L^{-1}_i\subset \mathcal E$.
If $a_{ij}=0$ for $i\neq j$, then we also write $q:=\langle a_{00},\dots , a_{r+1,r+1}\rangle$.

%

\begin{lemma} \label{lem:q=0:quadricbundle}
Let $k$ be an algebraically closed field of characteristic zero.
Let $S$ be a projective variety over $k$, and let $L_1,\dots ,L_{r+1}$ and $L$ be line bundles on $S$ such that $L_i\otimes L_j\otimes L$ is base point free for all $i$ and $j$.
Consider $\mathcal E=\bigoplus_{i=0}^{r+1} L_i^{-1}$ and let $q\in H^0(S,\Sym^2(\mathcal E^\vee)\otimes L)$ be a general section.
Then the hypersurface $X:=\{q=0\}\subset \CP(\mathcal E)$ satisfies the following.
\begin{enumerate}
\item If $S$ is smooth, then $X$ is smooth.
\item If $\binom{r+3}{2}>\dim(S)$, or $L_i\otimes L_j\otimes L$ is trivial for some $i$ and $j$, then $X\longrightarrow S$ is a quadric bundle, i.e.\ $X$ is flat over $S$ and the generic fibre is a smooth quadric.
\end{enumerate}
\end{lemma}
\begin{proof}
Consider the natural projection $\pi:\CP(\mathcal E)\longrightarrow S$.
Then, $\pi_\ast \mathcal O_{\CP(\mathcal E)}(k)=\Sym^k(\mathcal E^\vee)$ and so 
$$
H^0(S,\Sym^2(\mathcal E^\vee)\otimes L)\cong H^0(\CP(\mathcal E), \mathcal O_{\CP(\mathcal E)}(2)\otimes \pi^\ast L) 
$$
by the projection formula. 
To prove the first assertion, it thus suffices by Bertini's theorem in characteristic zero to see that $\mathcal O_{\CP(\mathcal E)}(2)\otimes \pi^\ast L$ is base point free.
The latter follows by considering the global sections $a_{ii}z_i^2$ for $i=0,\dots ,r+1$ and with varying $a_{ii}\in H^0(S,L_i^{\otimes 2}\otimes L)$, because $L_i^{\otimes 2}\otimes L$ is base point free by assumption.

It remains to prove the second item.
Since $L_i\otimes L_j\otimes L$ is base point free for all $i,j$ and $q$ is general, the generic fibre of $f:X\longrightarrow S$ is smooth.
For the same reason, 
$q_s\neq 0$ for all $s\in S$ if $\binom{r+3}{2}>\dim(S)$. 
If $L_i\otimes L_j\otimes L$ is trivial for some $i$ and $j$, then $q_s\neq 0$ for all $s$ is clear.
Hence, $X$ is flat over $S$ in either case, which proves the lemma.
\end{proof}

%

\begin{lemma} \label{lem:rationaldefotype}
Let $k$ be an algebraically closed field of characteristic zero.
Let $S$ be a smooth projective rational variety over $k$ and let $L_0,\dots,L_{r+1}$ and $L$ be line bundles on $S$ such that $L_i\otimes L_j\otimes L$ is base point free for all $i,j$.
Let $X$ be a smooth $r$-fold quadric bundle over $S$, given by a symmetric matrix $A=(a_{ij})$ of global sections $a_{ij}\in H^0(S,L_i\otimes L_j\otimes L)$ as in (\ref{eq:X:2}) above. 
If $r\geq \dim(S)$, then $X$ deforms to a smooth rational variety over $k$.

More precisely, if $r\geq \dim(S)$, $a_{mm}=0$ for some $0\leq m\leq r+1$, and 
 the remaining $a_{ij}$ are sufficiently general, 
 then the corresponding quadric bundle $X$ is smooth and rational.
\end{lemma}

\begin{proof} 
Since all quadric bundles of the given type are parametrized by some open subset of $H^0(S,\Sym^2 (\mathcal E^\vee) \otimes L)$, where $\mathcal E^\vee =\bigoplus_{i=0}^{r+1} L_i$, we see that they have all the same deformation type.
It thus suffices to prove that for general sections $a_{ij}\in H^0(S,L_i\otimes L_j\otimes L)$ with $a_{mm}=0$, $X$ is smooth; $X$ is then automatically rational because it admits a section.  
We may for simplicity assume $m=r+1$.
Considering the sections $a_{ii}z_i^2$ and using that $L_i^{\otimes 2}\otimes L$ is base point free for $i=0,\dots ,r$, Bertini's theorem shows that the only possible singularity of $X$ occurs at $z_0=\dots=z_{r}=0$, where we use the local chart (\ref{eq:X:2}). 
Using the Jacobian criterion, we see that a singular point of $X$ must lie on the fibre above a point of $S$ where $a_{r+1,i}$ vanishes for $i=0,\dots ,r$.  
Since $r\geq \dim(S)$, this locus is empty by our base point freeness assumption. 
This proves the lemma.
\end{proof}

For $S=\CP^n_k$, we have $L=\mathcal O(l)$ and $L_i=\mathcal O(l_i)$ for some integers $l,l_0,\dots ,l_{r+1}$.
The deformation type of $X$ as in (\ref{eq:X:2}) is then completely determined by the integers $d_i:=2l_i+l$ for $i=0,\dots ,r+1$, i.e.\ by the degrees of the diagonal entries of the matrix $A$ in (\ref{eq:X:2}).
This observation gives rise to the following definition.

\begin{definition}  \label{def:type}
Let $k$ be an algebraically closed field of characteristic different from two.
Let $r,n\geq 1$ and $l_0,\dots ,l_{r+1},l$ be integers.  
An $r$-fold quadric bundle $X$ over $\CP^n_k$, which is given by a symmetric matrix $A=(a_{ij})$ of homogeneous polynomials of degrees $|a_{ij}|=l_i+l_j+l$ as in (\ref{eq:X:2}), is called of type $(d_i)_{0\leq i\leq r+1}$ if $d_i=2l_i+l$. 
\end{definition}
 
We usually assume that $d_i\geq 0$ for all $i$.
This is justified by the observation that if $d_i<0$ for some $i$, then $a_{ii}=0$ and so $z_i=1$ and $z_j=0$ for $j\neq i$ yields a section of $X\longrightarrow \CP^n_k$.
Hence, $X$ is automatically rational in that case.

If $\operatorname{char}(k)=0$, a smooth quadric bundle of type $(d_i)_{0\leq i\leq r+1}$ over $\CP^n_k$ exists by Lemma \ref{lem:q=0:quadricbundle} if all $d_i$ are non-negative of the same parity and additionally one of the following holds: $\binom{r+3}{2}>n$ or $d_i=0$ for some $i$.

The following is an immediate consequence of Lemma \ref{lem:rationaldefotype}.

\begin{corollary} \label{cor:rational-defotype}
Let $k$ be an algebraically closed field of characteristic zero.
Let $n$ and $r$ be positive integers with $r\geq n$ and let $(d_i)_{0\leq i\leq r+1}$ be a tuple of non-negative integers of the same parity.
Then some smooth $r$-fold quadric bundles of type $(d_i)_{0\leq i\leq r+1}$ over $\CP^n_k$ are rational. 
\end{corollary}

The following two examples of quadric bundles are well-known. 

\begin{lemma} \label{lem:singhyper}
Let $k$ be an algebraically closed field of characteristic different from two.
Let $n,r$ be integers with $\binom{r+3}{2}>n>0$. 
Let $P\subset \CP^{n+r+1}_k$ be an $r$-plane, and let $ X\subset \CP^{n+r+1}_k$ be a general hypersurface  of degree $d$ with multiplicity $d-2$ along $P$.
Then, $X$ is birational to a general $r$-fold quadric bundle of type $(d-2,\dots ,d-2,d)$ over $\CP^n_k$.
\end{lemma}

\begin{proof} 
Choose coordinates $x_0,\dots ,x_{n},y_0,\dots ,y_{r}$ on $\CP^{n+r+1}_k$ such that $P=\{x_0=\dots =x_{n}=0\}$.
If $X=\{f=0\}$, then 
$$
f=\sum_{i,j=0}^{r}a_{ij}y_iy_j+\sum_{k=0}^{r}(a_{k,r+1}+a_{r+1,k})y_k+ a_{r+1,r+1} ,
$$
for some homogeneous polynomials $a_{ij}=a_{ji}$, $a_{k,r+1}=a_{r+1,k}$ and $a_{r+1,r+1}$ in $x_0,\dots ,x_{n}$ of degrees $d-2$, $d-1$ and $d$, respectively. 
We introduce an additional variable $y_{r+1}$ and homogenize the above equation with respect to the $y_i$'s.
This shows that the symmetric matrix $A=(a_{ij})_{0\leq i,j\leq r+1}$ corresponds to a general $r$-fold quadric bundle of type $(d-2,\dots ,d-2,d)$, which is clearly birational to $X$. 
(In fact, it is the blow-up $Bl_PX$.) 
\end{proof}

\begin{lemma}\label{lem:doublePn}
Let $k$ be an algebraically closed field of characteristic different from two, and let $n,r\geq 1$ be integers. 
Let $P\subset \CP^{n+r}_k$ be an $(r-1)$-plane, and let $D\subset \CP^{n+r}_k$ be a general hypersurface of even degree $d$ with multiplicity $d-2$ along $P$.
Then the double covering $X\stackrel{2:1}\longrightarrow \CP^{n+r}_k$, branched along $D$, is birational to a general $r$-fold quadric bundle of type $(0,d-2,\dots ,d-2,d)$ over $\CP^n_k$.
\end{lemma}

\begin{proof}
The double cover $X$ is given by $s^2=f$, where $D=\{f=0\}$.
Choosing coordinates $x_0,\dots ,x_n,y_1,\dots ,y_r$ of $\CP^{n+r}_k$, similarly as in the proof of Lemma \ref{lem:singhyper} shows that $X$ is birational to a quadric bundle over $\CP^n_k$ of type $(0,d-2,\dots ,d-2,d)$; the coordinate $s$ plays the role of $y_0$ in the proof of Lemma \ref{lem:singhyper}.
The corresponding symmetric matrix $A=(a_{ij})_{0\leq i,j\leq r+1}$ satisfies $a_{00}=1$ and $a_{i0}=0$ for $i\geq 1$; the remaining entries of $A$ are general.
Conversely, if $A=(a_{ij})_{0\leq i,j\leq r+1}$ is the symmetric matrix of a general $r$-fold quadric bundle of type $(0,d-2,\dots ,d-2,d)$, then $a_{00}$ is a nonzero constant and so we can transform $A$ into a symmetric matrix with $a_{00}=1$ and $a_{i0}=0$ for all $i\geq 1$.
This proves the lemma.
\end{proof}

\begin{proposition}[Voisin]\label{prop:voisin}
Let $d_0=0$, $d_1=d_2=2$ and $d_3=4$.
Let $W$ be the complex vector space of symmetric $4\times 4$-matrices $A=(a_{ij})_{0\leq i,j\leq 3}$ such that $a_{ij}\in \C[x_0,x_1,x_2]$ is homogeneous of degree $(d_i+d_j)/2$ with $a_{i0}=0$ for $i=1,2,3$.
Then the set of points in $\CP(W)$ which parametrize smooth quadric surface bundles of type $(0,2,2,4)$ over $\CP^2_\C$ with a rational section is dense in the analytic topology. 
\end{proposition}

\begin{proof} 
There is a Zariski open subset $B\subset \CP(W)$ which parametrizes smooth quadric surface bundles of type $(d_0,d_1,d_2 ,d_3)$ over $\CP^2_\C$.
There is a universal family $\pi:\mathcal X\longrightarrow B$.
As we have seen in Lemma \ref{lem:doublePn}, this family coincides with the universal family of (blow-ups of) double covers of $\CP^4_\C$, branched along a quartic hypersurface which is singular along a fixed line.
If the fibre $X_b$ above $b\in B$ admits a rational multisection of odd degree, then $X_b$ admits a rational section by Springer's theorem \cite{springer}.  
Since the integral Hodge conjecture is known for codimension two cycles on quadric surface bundles over surfaces (cf.\ \cite[Corollaire 8.2]{CT-voisin}), it suffices to show that the set of points $b\in B$ such that $X_b$ admits a Hodge class of type $(2,2)$ which intersects the general fibre of $X_b\longrightarrow \CP^2_\C$ in odd degree is dense in $B$.
The latter is proven in \cite[Proposition 2.4]{voisin2}, which is not affected by the gap (cf.\ \cite{beauville-erratum}); similar arguments have later been used in \cite{HPT} and \cite{HPT3}. 
\end{proof}

\section{The specialization method via weak decompositions of the diagonal} 
\label{sec:specialization}

Recall from Section \ref{subsec:intro:specialize} that we aim to generalize the method of Voisin \cite{voisin} and Colliot-Th\'el\`ene--Pirutka \cite{CT-Pirutka} to degenerations where the special fibre is allowed to have quite arbitrary singularities and where an explicit resolution of those can be avoided.
The first step is the following small but crucial improvement of the original technique in \cite{voisin} and \cite{CT-Pirutka}; 
the proof is inspired by  \cite{voisin,CT-Pirutka},
Totaro's paper \cite{totaro} and the original arguments of Bloch and Srinivas.

\begin{proposition}\label{prop:degeneration}
Let $R$ be a discrete valuation ring with fraction field $K$ and residue field $k$, with $k$ algebraically closed.
Let $\pi: \mathcal X\longrightarrow \Spec R$ be a flat proper scheme of finite type over $R$ 
 with geometrically integral fibres. 
Let  $Y:=\mathcal X\times k$ be the special fibre and suppose that there is a resolution of singularities $\tau:\widetilde {Y}\longrightarrow Y$ with the following properties:
\begin{enumerate}
\item for some Rost cycle module $M^\ast$ over $k$, there is an unramified class $\alpha \in M^i_{nr}(\widetilde {Y})$ which is nontrivial, i.e.\ $\alpha\notin \im (M^i(k)\longrightarrow M^i_{nr}(\widetilde Y))$; 
 \label{item:prop:alpha} 
\item there is an open subset $U\subset Y$ such that $\tau^{-1}(U)\longrightarrow U$ is universally $\CH_0$-trivial, and such that each irreducible component $E_i$ of $\widetilde Y\setminus \tau^{-1}(U)$ is smooth and the restriction of $\alpha$ to $E_i$ is trivial.  
\end{enumerate}
Then, no resolution of singularities of the geometric generic fibre $X:=\mathcal X\times \overline K$ admits an integral decomposition of the diagonal.
\end{proposition}

The assumptions on the resolution $\tau$ in Proposition \ref{prop:degeneration} are weaker and easier to check than those in \cite[Theorem 2.1]{voisin} and \cite[Th\'eor\`eme 1.14]{CT-Pirutka}.
Roughly speaking, instead of a universally $\CH_0$-trivial resolution of $Y$, we ask for a resolution which is universally $\CH_0$-trivial only over some open subset $U\subset Y$ and such that $\alpha$ restricts to zero on the complement.
In this paper we will mostly use the special case where $\tau^{-1}(U)\cong U$ is an isomorphism and so $\CH_0$-triviality is automatic. 
The idea  is to replace the Chow theoretic condition on the resolution $\tau$ from \cite{CT-Pirutka} by a cohomological one ($\alpha |_{E_i}$ is trivial), which is typically much more accessible. 

\begin{proof}[Proof of Proposition \ref{prop:degeneration}]
It suffices to prove that there is an algebraically closed field $F$ which contains $K$ and such that some resolution of $X\times F$ does not admit an integral decomposition of the diagonal.
Up to replacing $R$ by its completion (which does not change the residue field), we may thus assume that $R$ is a complete discrete valuation ring. 
For a contradiction, we assume that some resolution of $X$ admits an integral decomposition of the diagonal.
Pushing forward to $X$, we obtain a decomposition
\begin{align} \label{eq:DeltaX}
\Delta_{X}=[X \times z_X]+ B_X,
\end{align}
where $z_X \in \CH_0( X)$ is a zero-cycle of degree one, and where $\supp( B_X)\subset S_X \times X$ for some proper closed subset $ S_X \subsetneq  X$.
Since $k=\overline k$, the specialization homomorphism on Chow groups \cite[Example 20.3.5]{fulton} gives a decomposition of the diagonal of $Y$:
\begin{align} \label{eq:DeltaY}
\Delta_{Y}=[Y\times z]+B_Y,
\end{align}
where $z$ is a zero-cycle of degree one on $Y$, and where $\supp(B_Y)\subset S_Y\times Y$ for some proper closed subset $S_Y\subsetneq Y$.

Let $\widetilde U:=\tau^{-1}(U)$ and $E:=\widetilde Y\setminus \widetilde U$.
By assumptions, $\widetilde U\longrightarrow U$ is universally $\CH_0$-trivial.
Hence, for any field extension $L$ of $k$, the localization exact sequence \cite[Proposition 1.8]{fulton} gives the following commutative diagram, with exact rows:
\begin{align*}
\xymatrix{
\CH_0(E\times L) \ar[r] &\CH_0(\widetilde {Y}\times L) \ar[d]^{\tau_\ast} \ar[r] & \CH_0(\widetilde U\times L) \ar[d]^{\cong} \ar[r] & 0\\
&\CH_0(Y \times L)  \ar[r] & \CH_0(U\times L) \ar[r] & 0} 
\end{align*}
We apply this to $L=k(Y)$ and think about $\widetilde Y \times L$ and $ Y \times L$ as generic fibres of the projections $\pr_1:\widetilde Y\times \widetilde Y\longrightarrow \widetilde Y$ and $\pr_1: Y\times Y\longrightarrow Y$ to the first factors, respectively. 
We claim that this gives rise to a decomposition
\begin{align} \label{eq:Delta=B+C}
\Delta_{\widetilde {Y}}=[\widetilde {Y}\times \tilde z]+B+C,
\end{align}
where $\tilde z\in \CH_0(\widetilde {Y})$ has degree one (and maps to $z$),  $\supp(C)\subset \widetilde {Y}\times E$ and $\supp(B)\subset S\times \widetilde {Y}$, for some proper closed subset $S\subsetneq \widetilde{Y}$.  
Indeed, since $k=\overline k$, we may choose a lift $\tilde z$ of $z$ and then the above diagram together with (\ref{eq:DeltaY}) shows that the image of $\Delta_{\widetilde {Y}}-[\widetilde {Y}\times \tilde z] $ in $\CH_0(\widetilde {Y}\times L)$ restricts to zero on $\widetilde U\times L$, where $L=k(Y)$.
This yields (\ref{eq:Delta=B+C}), as claimed.

Letting the correspondence (\ref{eq:Delta=B+C}) act by pull-back gives an action
$$
\Delta_{\widetilde {Y}}^{\ast}=[\widetilde {Y}\times \tilde z]^\ast+B^\ast+C^\ast:M^i_{nr}(\widetilde Y)\longrightarrow M^i_{nr}(\widetilde Y) ,
$$
which is the identity because $\Delta_{\widetilde {Y}}$ is the class of the diagonal.
As recalled in Section \ref{subsec:Rostcycle}, $B^\ast$ acts trivially because $B$ does not dominate the first factor.
Moreover, for each closed point $y\in \widetilde{Y}$, $[\widetilde {Y}\times  y]^\ast$ factors through $M^i_{nr}(y)=M^i(k)$ and the induced map $M^i(k)\longrightarrow M^i_{nr}(\widetilde {Y})$ is the natural one.
The image of $[\widetilde {Y}\times \tilde z]^\ast$ is therefore contained in the subgroup of trivial unramified elements $M^i(k)\subset M^i_{nr}(\widetilde {Y})$, where we use that  $M^i(k)\longrightarrow M^i_{nr}(\widetilde {Y})$ is injective because $\widetilde Y$ has a rational point and denote its image by $M^i(k)\subset M^i_{nr}(\widetilde {Y})$.
The above decomposition of the diagonal thus shows that, up to trivial unramified elements from $M^i(k)$, we have
$
\alpha=C^\ast (\alpha)
$.

We may write $C=\sum_i C_i$, where $\supp(C_i)\subset \widetilde {Y}\times E_i$, and where the $E_i$ denote the irreducible components of $E$. 
Since $E_i$ is smooth, $C_i^\ast:M^i_{nr}(\widetilde{Y})\longrightarrow M^i_{nr}(\widetilde{Y})$ 
factors through the restriction map $M^i_{nr}(\widetilde{Y})\longrightarrow M^i_{nr}(E_i)$. 
Our assumptions therefore imply $C_i^\ast ( \alpha) \in M^i(k)\subset M^i_{nr}(\widetilde {Y})$ for all $i$.
This implies $\alpha\in M^i(k) \subset M^i_{nr}(\widetilde {Y})$, which contradicts our assumption that $\alpha$ is nontrivial.
This finishes the proof of the proposition.
\end{proof}

\begin{remark}
The unramified cohomology group $M^i_{nr}$ in item (\ref{item:prop:alpha}) of Proposition \ref{prop:degeneration} can be replaced by any other birational invariant on which integral correspondences act similarly.  
For instance,  Proposition \ref{prop:degeneration} remains true if we replace condition (\ref{item:prop:alpha}) by the existence of a nontrivial differential form $\alpha \in H^0(\widetilde {Y},\Omega_{\widetilde {Y}}^i)$ for some $i\geq 1$, cf.\ \cite{totaro}. 
\end{remark}

\section{A vanishing result}
\label{sec:vanishing}

If the special fibre $Y$ in Proposition \ref{prop:degeneration} is birational to a quadric bundle over $\CP^n_\C$ whose generic fibre is a quadric of CTO type, then condition (\ref{item:prop:alpha}) of Proposition \ref{prop:degeneration} is satisfied by Proposition \ref{prop:CTO}.
In this section we establish a vanishing result which ensures that under some mild assumptions, also the second condition in Proposition \ref{prop:degeneration} is satisfied.

Recall that for any dominant rational map $f:Y\dashrightarrow S$, there is a generic fibre $Y_\eta$ over the function field of $S$, well-defined up to birational equivalence.
An explicit representative of $Y_\eta$ is given by the generic fibre of $f|_U :U\longrightarrow S$, where $U\subset Y$ is some open dense subset on which $f$ is defined.

\begin{proposition}\label{prop:alpha':CTO}
Let $Y$ be a normal complex projective variety and let $S$ be a normal complex projective rational $n$-fold for some $n\geq 2$.
Let $f:Y\dashrightarrow S$ be a dominant rational map whose generic fibre $Y_\eta$ is stably birational to a CTO type quadric $Q$ over $K=\C(S)$, defined by a neighbour of the Pfister form $\langle\langle a_1,\dots ,a_{n-1},b_1b_2\rangle\rangle$, for some $a_i,b_j\in K^\ast$. 
Set $\alpha_j:=(a_1,\dots ,a_{n-1},b_j)\in H^n(K,\mu_2^{\otimes n})$ and let $\alpha':=f^\ast \alpha_1\in H^n_{nr}(\C(Y)\slash \C,\mu_2^{\otimes n})$ be the unramified class from Proposition \ref{prop:CTO}.
Then the following holds:
\begin{enumerate}
\item[($\ast\ast$)] for any prime divisor $E\subset Y$ which does not dominate $S$, the restriction of $\alpha'$ to $E$ vanishes:  
$
\alpha'|_E=0\in H^n(\C(E),\mu_2^{\otimes n}). 
$ 
\end{enumerate} 
\end{proposition}

We will use the following lemma, which reformulates \cite[Propositions 1.4 and 1.7]{merkurjev} in geometric terms.

\begin{lemma}\label{lem:merkurjev}
Let $f:Y\dashrightarrow S$ be a dominant rational map between normal complex projective varieties.
Let $y\in Y^{(1)}$ be a codimension one point.
Then there is a normal projective model $S'$ of $S$, 
such that the induced rational map $f':Y\dashrightarrow S'$ maps $y$ either to the generic point of $S'$ or to the generic point of a divisor on $S'$.
\end{lemma} 

\begin{proof}
Since $Y$ is normal, $f$ is defined at $y$.
If $f(y)$ is dense in $S$, then $S'= S$ works. 
Otherwise, the valuation $\nu\in \Val(\C(Y)\slash \C)$ induced by $y\in Y^{(1)}$ restricts to a geometric discrete valuation $\mu$ of rank one on $f^\ast (\C(S))\subset \C(Y)$, see \cite[Proposition 1.4]{merkurjev}.
By \cite[Proposition 1.7]{merkurjev}, $\mathcal O_\mu=\mathcal O_{S',s'}$ for some normal projective variety $S'$, birational to $S$, and some codimension one point $s'\in (S')^{(1)}$.
The induced dominant rational map $f':Y\dashrightarrow S'$ sends $y$ to $s'$. 
This proves the lemma. 
\end{proof} 

\begin{proof}[Proof of Proposition \ref{prop:alpha':CTO}]
Let $y\in Y^{(1)}$ be the generic point of $E$.
Since $Y$ is normal, $f$ is defined at $y$.
By Lemma \ref{lem:merkurjev}, we may up to replacing $S$ by some different normal projective model assume that $x:=f(y)$ is a codimension one point on $S$.
Consider the discrete valuation rings $A:=\mathcal O_{S,x}$ and $B:= \mathcal O_{ Y,y}$ and note that $f$ induces an injection $A\hookrightarrow B$.
By the definition of CTO type quadrics, there is some $j\in \{1,2\}$ with $\del^n_A \alpha_j=0$.

The generic fibre of $f$ is stably birational to a CTO type quadric associated to a neighbour of the Pfister form $\langle\langle a_1,\dots ,a_{n-1},b_1b_2\rangle\rangle$.
Since unramified cohomology is a stable birational invariant \cite{CTO}, 
we conclude $f^\ast \alpha_1=f^\ast \alpha_2\in H^n_{nr}(\C(Y)\slash \C,\mu_2^{\otimes n})$ from Theorem \ref{thm:OVV}. 
It thus suffices to prove that $f^\ast \alpha_j$ restricts to zero on $E$, where $j$ is as above.
Since $\del^n_A \alpha_j=0$, 
$$
\alpha_j \in H^n_{\text{\'et}}(\Spec A,\mu_2^{\otimes n})\subset H^n(K,\mu_2^{\otimes n}) ,
$$
see \cite[\S 3.3 and \S 3.8]{CT}.
Functoriality of \'etale cohomology yields a commutative diagram
$$
\xymatrix{
H^n_{\text{\'et}}(\Spec A,\mu_2^{\otimes n}) \ar[r] \ar[d]^{f^\ast} & H^{n}(\kappa(x),\mu_2^{\otimes n}) \ar[d]^{f^\ast} \\
H^n_{\text{\'et}}(\Spec B,\mu_2^{\otimes n}) \ar[r]&H^{n}(\C(E),\mu_2^{\otimes n}) ,
}
$$
where the vertical arrows are induced by restriction to the corresponding closed points, respectively.
Since $H^{n}(\kappa(x),\mu_2^{\otimes n})=0$ by Theorem \ref{thm:milne}, $\alpha'|_E=0$ follows from the commutativity of the above diagram.
This finishes the proof of the proposition. 
\end{proof}

\section{Existence of CTO type quadrics in arbitrary dimensions}
In this section, we aim to prove that CTO type quadrics (see Section \ref{subsec:CTO}) exist 
over $\C(\CP^n)$ for arbitrary $n\geq 2$. 

\subsection{Construction of quadrics over $\C(\CP^n)$ via arrangements of quadrics in $\CP_\C^n$}
We choose coordinates $x_0,\dots ,x_n$ on $\CP^n_\C$.
For $i=1,\dots , n-1$ we consider homogeneous polynomials $h_i\in \C[x_0,\dots ,x_{n}]$ of degree two and define
\begin{align} \label{def:ai}
a_i:=\frac{h_i}{x_0^2}
\end{align}
for $i=1,\dots ,n-1$.
In order to obtain a candidate CTO type quadric, we need to define two more rational functions $b_1$ and $b_2$, which we  will do next.

Choose two homogeneous polynomials $g_{10},g_{20}\in \C[x_0,\dots ,x_{n}]$ of degree two.
For any  $\epsilon=(\epsilon_1,\dots ,\epsilon_{n-1})\in I:=\{0,1\}^{n-1}$ and any $j=1,2$, we then consider 
$$
g_{j \epsilon}:=g_{j0}+\sum_{i=1}^{n-1} \epsilon_i h_i .
$$
With this definition, we put
\begin{align} \label{def:gi}
g_1:=\prod_{\epsilon \in I} g_{1 \epsilon} \ \ \text{and}\ \ g_2:=\prod_{\epsilon\in I}g_{2 \epsilon} .
\end{align}
Finally, let $N:=2^{n}$ and define
\begin{align} \label{def:bj}
b_1:=\frac{g_1}{x_0^{N}}\ \ \text{and}\ \ b_2:=\frac{g_2}{x_0^{N}}.
\end{align}

\subsubsection{Assumptions}
In the above construction, we will always assume that the homogeneous degree two polynomials $h_i$ and $g_{j0}$ satisfy the following assumptions.

For any $1\leq i_{1}<\dots <i_c \leq n-1$ with $c\geq 0$, the following holds for  $j=1,2$:
\begin{align}  \label{eq:hi=g1=0}
\codim_{\CP_\C^n} (\{h_{i_1} &=\dots =h_{i_c}=g_j=0\}) \geq c+1 , \\
\codim_{\CP_\C^n} (\{h_{i_1} &=\dots =h_{i_c}=g_1=g_2=0\}) \geq c+2 . \label{eq:hi=g1=g2=0} 
\end{align}
Moreover, we will assume that for $j=1,2$, the following symbol is nonzero
\begin{align}
0\neq (a_1,\dots ,a_{n-1},b_j)\in H^{n}(\C(\CP^n),\mu_2^{\otimes n}). \label{item:ai-symbol}
\end{align}

\subsubsection{Existence} \label{subsec:existence}
%
Let $l_1,\dots ,l_{2n+2}\in \C[x_0,\dots ,x_{n}]$ be linear homogeneous polynomials which are general subject to the condition\footnote{Condition (\ref{eq:li}) is not essential; it will only be used later in the proof of density of the rational fibres in the family of Theorem \ref{thm:defo}.} that 
\begin{align} \label{eq:li}
l_{1},l_{2},l_3,l_4 \in \C[x_0,x_1,x_2] .
\end{align} 
We put 
\begin{align}\label{def:hi,gj}
h_i:=l_{2i-1}l_{2i}\ \ \text{and}\ \ g_{j0}:=l_{2n-3+2j}l_{2n-2+2j}.
\end{align}
Conditions (\ref{eq:hi=g1=0}) and (\ref{eq:hi=g1=g2=0}) are then clearly satisfied; in fact, (\ref{eq:hi=g1=0}) and (\ref{eq:hi=g1=g2=0}) follow from 
$$
\{h_1=\dots =h_{n-1}=g_1=g_2\}=\emptyset ,
$$
which holds by our genericity assumption on the $l_i$.
To see that also (\ref{item:ai-symbol}) holds, it suffices by symmetry to deal with the case $j=1$.
To prove our claim, we take successive residues of $(a_1,\dots ,a_{n-1},b_1)$ along $\{l_{2i}=0\}$ for $i=n,n-1,\dots ,2,1$.
Using Lemma \ref{lem:residue} and our genericity assumptions on the $l_i$'s, we end up with the nontrivial element of $H^0(\Spec \C, \mu_2)$.
This proves (\ref{item:ai-symbol}).

We have thus proven that the choice of $h_i$ and $g_{j0}$ as in (\ref{def:hi,gj}) and the resulting $a_i$ and $b_j$ given by (\ref{def:ai}) and (\ref{def:bj}), satisfy all our assumptions (\ref{eq:hi=g1=0}), (\ref{eq:hi=g1=g2=0}) and (\ref{item:ai-symbol}).

\subsubsection{Key Property} 
Besides (\ref{eq:hi=g1=0})--(\ref{item:ai-symbol}), the most important property of this construction is as follows. 
Let $g_1$ and $g_2$ be as in (\ref{def:gi}), then, for any $i=1,\dots ,n-1$, 
\begin{align} \label{eq:gi=square}
\text{ the image of $g_1$ and $g_2$ in $\C[x_0,\dots ,x_{n}]\slash (h_i)$ becomes a square.}
\end{align}

\subsubsection{Remarks}
The above construction is inspired by  \cite[Exemple 2.4 and 3.3]{CTO}, where examples of CTO type quadrics for $n=2,3$ are given.
While that construction yields as degeneration divisor a special configuration of hyperplanes (cf.\ \cite[p.\ 150, Fig.\ 2]{CTO}), our construction relies on a configuration of pairs of hyperplanes given by $h_i=0$ and quadrics given by $g_{j\epsilon}=0$.  
Already for $n=3$, our  construction yields smaller bounds on the total degree of the degeneration divisor. 
This is important  in view of applications such as Theorem \ref{thm:type} and Corollary \ref{cor:hyper}, stated in the introduction, where small bounds on the degrees are desirable.

\subsection{Proof of existence -- a key result} 
In this section we prove that the above construction yields quadrics of CTO type. 
To this end, we do not follow the original approach of Colliot-Th\'el\`ene and Ojanguren. 
In fact, we do not try to generalize \cite[Compl\'ement 3.2]{CTO}, because we were unable to see how to split the argument according to the dimension of the center $x\in \CP^n_\C$ of the valuation $\nu\in \Val(\C(\CP^n)\slash \C)$ for arbitrary $n$; that strategy had however been used in all previous geometric constructions of  quadrics with nontrivial unramified cohomology we are aware of, cf.\ \cite{CTO,Pirutka,HPT}. 

\begin{proposition} \label{prop:examplesofCTO-type}
Let $n\geq 2$ and let $a_i,b_j\in K=\C(\CP^n)$ be as in (\ref{def:ai}) and (\ref{def:bj}).
Suppose that the assumptions (\ref{eq:hi=g1=0}), (\ref{eq:hi=g1=g2=0}) and (\ref{item:ai-symbol}) hold.
Then any Pfister neighbour of $\langle\langle a_1,\dots ,a_{n-1},b_1b_2 \rangle\rangle $ defines a CTO type quadric over $K$.  
\end{proposition}

\begin{proof}
By (\ref{item:ai-symbol}), the class $\alpha_j:=(a_1,\dots ,a_{n-1},b_j)\in H^n(K,\mu_2^{\otimes n})$ is nonzero.
It thus suffices to prove that for each $\nu\in \Val(K\slash \C)$, $\del_\nu^n\alpha_j=0$ for $j=1$ or $2$.

To prove this, let $\nu\in \Val(K\slash \C)$.
We can choose a normal complex projective variety $S$ together with a proper birational morphism $f:S\longrightarrow \CP^n_\C$, such that $\nu$ corresponds to a codimension one point $s\in S^{(1)}$.
Let $x:=f(s)\in \CP^n_\C$ be its image on $\CP^n_\C$.

By construction, $x$ is a point of codimension at least one on $\CP^n_\C$.
Hence, there is some $i$ with $x_i(x)\neq 0$. 
Multiplying the first $n-1$ entries of $\alpha_j$ by $x_0^2/x_i^2$ and the last entry by $x_0^{2^n}/x_i^{2^n}$ (which does not change the cohomology class $\alpha_j$), we see that we may without loss of generality assume $x_0(x)\neq 0$.
In particular, $a_i$ and $b_j$ are regular functions locally at $x$ and we may from now on work on an affine open subset where $x_0\neq 0$. 

We consider the completion $\widehat {\mathcal O_{S,s}}$ and let $\widehat K:=\Frac(\widehat {\mathcal O_{S,s}})$ be its field of fractions. 
By  Lemma \ref{lem:residue:compatible}, the residue $\del_\nu^n=\del_{\mathcal O_{S,s}}^n$ fits into a commutative diagram
\begin{align} \label{diag:residue}
\xymatrix{
H^n(\widehat K,\mu_2^{\otimes n})\ar[r]^-{\del_{\widehat {\mathcal O_{S,s}}}^n} &H^{n-1}(\kappa(s),\mu_2^{\otimes (n-1)}) \\
H^n(K,\mu_2^{\otimes n}) \ar[r]^-{\del_\nu^n} \ar[u] & H^{n-1}(\kappa(s),\mu_2^{\otimes (n-1)}) \ar[u]_{\id} .
}
\end{align}   

To prove the proposition, we need to show that $\del_\nu^n\alpha_j=0$ for $j=1$ or $2$.
We divide the argument into three cases. 

\textbf{Case 1.} $g_1(x)\neq 0$.

If $h_i(x)\neq 0$ for all $i$, then $\del_\nu^n \alpha_1=0$ by Lemma \ref{lem:residue}.
On the other hand, if at least one $h_i$ vanishes at $x$, then (\ref{eq:gi=square}) implies that $g_1$ becomes a nontrivial square in the residue field $\kappa(s)$.
By Hensel's lemma, $g_1$ becomes a square in the completion $\widehat K$, and so $\del_\nu^n\alpha_1=0$ by the commutative diagram (\ref{diag:residue}).
This concludes Case 1.

\textbf{Case 2.} $g_1(x)=0$ and $h_i(x)\neq 0$ for all $i=1,\dots ,{n-1}$.

In this case, we consider $\alpha_2$.
If $g_2(x)\neq 0$, then $\del_\nu^n \alpha_2=0 $ by Lemma \ref{lem:residue}.
On the other hand, since $g_1(x)=0$ by assumptions, $g_2(x) = 0$ implies by (\ref{eq:hi=g1=g2=0}) that $x$ has codimension at least two in $\CP^n_\C$.
Moreover, $\del_\nu^n\alpha_2$  is a multiple of $(a_1,\dots ,a_{n-1})$ by Lemma \ref{lem:residue}, where by slight abuse of notation we do not distinguish between $a_i$ and its image in $\kappa(s)$.
But this shows that the residue $\del_\nu^n \alpha_2$ is a pullback of a class from $H^{n-1}(\kappa(x),\mu_2^{\otimes (n-1)})$ and so it must vanish by Theorem \ref{thm:milne} because $x$ has dimension at most $n-2$.

\textbf{Case 3.} $g_1(x)=0$ and $h_i(x)= 0$ for some $i=1,\dots ,{n-1}$.

Consider $\alpha_2=(a_1,\dots ,a_{n-1},b_2)$ and suppose that exactly $c$ entries of $\alpha_2$ vanish at $x$.
Since $g_1(x)=0$, (\ref{eq:hi=g1=0}) and (\ref{eq:hi=g1=g2=0}) imply that $x$ has codimension at least $c+1$.
By assumptions $c\geq 1$ and so Lemma \ref{lem:residue} shows that we can write
$$
\del_\nu^n\alpha_2=\beta\cup \gamma ,
$$
where $\beta \in H^{c-1}(\kappa(s),\mu_2^{\otimes (c-1)})$ and $\gamma$ is a symbol of degree $n-c$ which is given by (the images of) all entries of $\alpha_2$ which do not vanish at $x$.
In particular, $\gamma$ comes from a class of $H^{n-c}(\kappa(x),\mu_2^{\otimes (n-c)})$ and so it vanishes because $x$ is a point of dimension at most $n-c-1$. 
This concludes Case 3. 

Cases 1, 2 and 3 above finish the proof of the proposition.
\end{proof}

\begin{remark}
The above proof did not use that  the $h_i$ have degree two.
In fact, we can start with any collection of homogeneous polynomials $h_i,g_{j0}\in \C[x_0,\dots ,x_n]$ of the same even degree $2m$.
We may then define $g_1$ and $g_2$ as in (\ref{def:gi}) and put $a_i=h_i/x_0^{2m}$ and $b_j=g_j/x_0^{(2m)^n}$.
The proof of Proposition \ref{prop:examplesofCTO-type} shows then that the $a_i$ and $b_j$ define CTO type quadrics as soon as the assumptions (\ref{eq:hi=g1=0}), (\ref{eq:hi=g1=g2=0}) and (\ref{item:ai-symbol}) hold.
\end{remark}

\subsection{Quadric bundles of CTO type and some estimates}
Here we construct and analyse some quadric bundles whose generic fibres are quadrics of CTO type.

We will need a suitable bijection between $ \{0,1\}^{n}$ and $\{0,1,\dots ,2^n-1\}$. 
We start with $I=\{0,1\}^{n-1}$ and define the length of an element $\epsilon\in I$ by $|\epsilon|=\sum\epsilon_i$.  
We then choose any bijection $\phi':I \stackrel{\sim}\longrightarrow \{0,\dots ,2^{n-1}-1\}$ with $\phi'(\epsilon)\leq \phi'(\epsilon')$ if $|\epsilon|\leq |\epsilon'|$.
With this in mind, we define 
$$
\phi: \{0,1\}^{n}=I\times \{0,1\}\stackrel{\sim}\longrightarrow \{0,1,\dots ,2^n-1\},\ \ (\epsilon,\epsilon_n)\mapsto  \phi'(\epsilon)+\epsilon_n2^{n-1}  .
$$ 

\begin{definition} \label{def:cici'etc}
Let $n\geq 2$, 
and let $l_1,\dots ,l_{2n+2}\in \C[x_0,x_1,\dots ,x_n]$ be linear homogeneous polynomials as in \ref{subsec:existence}.
Equations (\ref{def:gi}) and (\ref{def:hi,gj}) then give two homogeneous polynomials $g_1$ and $g_2$ of degree $2^n$.  
For  $\epsilon \in \{0,1\}^n$, let
$$
c_\epsilon:=    \left( \prod_{i=1}^{n-1} (l_{2i-1}l_{2i})^{\epsilon_i} \right)  (g_1g_2) ^{\epsilon_{n}} .
$$
Let $\phi: \{0,1\}^n \longrightarrow \{0,\dots ,2^n-1\}$ be the bijection from above.
Then we define the following homogeneous polynomials for $i=0,\dots ,2^{n}-1$:
\begin{enumerate}
\item $c_i:=c_{\phi^{-1}(i)}$; \label{item:def:ci:1}
\item $c_i':=l_1c_i$ if $l_1$ does not divide $c_i$ and $c_i':=l_1^{-1}c_i$ otherwise;   \label{item:def:ci:2} 
\end{enumerate} 
Moreover, we denote the degrees of the above  homogeneous polynomials by $m_i:=|c_i|$ and $m_i':=|c_i'|$, respectively.
\end{definition}

For later use, we will assume that the bijection $\phi'$ from above is chosen in such a way that the following holds for $n\geq 3$:
\begin{align}\label{eq:phi}
c_1=l_1l_2,\ \ c_2=l_3l_4 \ \ \text{and}\ \ c_{n}=l_1l_2l_3l_4 .
\end{align}

In the next definition, we consider the polynomials $\tilde c_i$ that are obtained by starting with $l_1l_3\cdots l_{2n-3}g_1 c_i$ and absorbing all squares which arise.
The homogeneous polynomials $\tilde c_i$ obtained this way are balanced, in the sense that $|\tilde c_i|=2^n+n-1$ for all $i$.
The formal definition is as follows.

\begin{definition} \label{def:citilde}
In the notation of Definition \ref{def:cici'etc}, and for  $\epsilon \in \{0,1\}^n$, let
$$
\tilde c_\epsilon:=    \left( \prod_{i=1}^{n-1} l_{2i-1}^{1-\epsilon_i}l_{2i}^{\epsilon_i} \right)  g_1^{1-\epsilon_n}g_2 ^{\epsilon_{n}} .
$$
Let $\phi: \{0,1\}^n \longrightarrow \{0,\dots ,2^n-1\}$ be the bijection from above.
Then we define the following homogeneous polynomials for $i=0,\dots ,2^{n}-1$:
\begin{enumerate} 
\item $\tilde c_i:=\tilde c_{\phi^{-1}(i)}$;  \label{item:def:ci:3}
\item $\tilde c_i':=l_1\tilde c_i$ if $l_1$ does not divide $\tilde c_i$ and $\tilde c_i':=l_1^{-1}\tilde c_i$ otherwise.  \label{item:def:ci:4}
\end{enumerate} 
Moreover, we denote the degrees of the above homogeneous polynomials by $\tilde m_i:=|\tilde c_i|$ and $\tilde m_i':= |\tilde c_i'|$, respectively.
\end{definition}

With the above definitions, we have the following corollary of Proposition \ref{prop:examplesofCTO-type}.

\begin{corollary} \label{cor:CTO-bundles:1}
Let $n\geq 2$ and $r$ be integers with $2^{n-1}-1\leq r\leq 2^n-2$. 
In the notation of Definitions \ref{def:cici'etc} and \ref{def:citilde}, the following quadratic forms
\begin{align*}
\langle c_0,\dots ,c_{r+1}\rangle, \ \ 
\langle c'_0,\dots ,c'_{r+1}\rangle,\ \ \langle \tilde c_0,\dots ,\tilde c_{r+1}\rangle \  \ \text{and}\ \ \langle \tilde c'_0,\dots ,\tilde c'_{r+1}\rangle ,
\end{align*}
 define hypersurfaces $Y\subset \CP(\mathcal E)$, where $\mathcal E=\bigoplus_{i=0}^{r+1}\mathcal O_{\CP^n}(-k_i)$ with 
 $$
 k_i=\lfloor m_i/2\rfloor ,\ \ k_i=\lfloor m_i'/2 \rfloor,\ \ k_i=\lfloor \tilde m_i/2 \rfloor\  \ \text{and}\ \ k_i=\lfloor \tilde m_i'/2 \rfloor ,
 $$ 
 respectively, such that the generic fibre of $Y\longrightarrow \CP^n_\C$ is a quadric of CTO type.
 Moreover, the hypersurface $Y$ associated to $\langle c_0,\dots ,c_{r+1}\rangle$ is flat over $\CP^n_\C$, i.e.\ it is a quadric bundle. 
\end{corollary}

\begin{proof} 
The given forms  are line bundle valued quadratic forms on $\mathcal E$ with values in $\mathcal O_{\CP^n_\C}$ or $\mathcal O_{\CP^n_\C}(1)$, depending on whether the entries of the given form have even or odd degrees, cf.\ Section \ref{subsec:quadricbundles}.
Since the entries of the given quadratic forms are nonzero and have no common factor, they define an integral hypersurface $Y\subset \CP(\mathcal E)$ whose generic fibre over $\CP^n_\C$ is a smooth quadric.
By the construction of $c_i$, $c_i'$, $\tilde c_i$ and $\tilde c_i'$, the forms \begin{align*}
\langle c_0,\dots ,c_{r+1}\rangle, \ \ 
\langle c'_0,\dots ,c'_{r+1}\rangle,\ \ \langle \tilde c_0,\dots ,\tilde c_{r+1}\rangle \  \ \text{and}\ \ \langle \tilde c'_0,\dots ,\tilde c'_{r+1}\rangle ,
\end{align*}
are similar to each other.
Since $2^{n-1}-1\leq r\leq 2^n-2$, each of the above forms is thus a Pfister neighbour of the Pfister form $\langle c_0,c_1,\dots ,c_{2^n-1} \rangle$. 
It then follows from Proposition \ref{prop:examplesofCTO-type} that the generic fibre of $Y\to \CP^n_\C$ is of CTO type.
The fact that $\langle c_0,\dots ,c_{r+1}\rangle$ defines a quadric bundle follows from $c_0=1$ and so $Y$ is flat over $\CP^n_\C$ in this case.
This proves the corollary.
\end{proof}

The following lemma gives some useful estimates for the degrees of the polynomials which appeared in the above corollary. 

\begin{lemma} \label{lem:bounds-degree}
Let $n\geq 2$. 
In the notation of Definitions \ref{def:cici'etc} and \ref{def:citilde}, the following holds:  
 \begin{enumerate}   
 \item $m_0=0$, $m_1=2$ and $m'_0=m_1'=1$;
  \item $\tilde m_i = 2^n+n-1$ for all $i$; \label{item:lem:bounds-degree:mi}
  \item $\tilde m_i' \leq \tilde m_i+1$ for all $i$; \label{item:lem:bounds-degree:mi'}
  \item $\sum_{i=0}^{r+1}  m_i= (r+2)(n+r+1) $ if $r=2^{n}-2$; \label{item:lem:nounds-degree:totaldegree} 
      \item $\sum_{i=0}^{r+1}  m_i \leq 2(r+1)(n+r)$ for all $2^{n-1}\leq r+1<2^n$;
      \label{item:lem:bounds-degree:Ms}
    \item $\sum_{i=0}^{r+1}  m_i'\leq 2(r+1)(n+r)$ for all $2^{n-1}\leq r+1<2^n$. \label{item:lem:bounds-degree:Ms'}
 \end{enumerate}
\end{lemma}

\begin{proof}
The first three items are clear.
Item (\ref{item:lem:nounds-degree:totaldegree}) follows from
$$
\sum_{i=0}^{2^n-1}  m_i= |(l_1l_2\cdots l_{2n-2}g_1g_2)^{2^{n-1}}|=2^{n-1}(2^{n+1}+2n-2)=2^n(2^n+n-1).
$$ 

We next aim to prove (\ref{item:lem:bounds-degree:Ms}).
If $r=2^n-2$, then it follows from (\ref{item:lem:nounds-degree:totaldegree}) and so it suffices to treat the case $r\leq 2^n-3$. 
Let thus $r$ be an integer with $2^{n-1}\leq r+1<2^n-1$ and write $r+1=2^{n-1}-1+j$ for some positive integer $j$.
Then we have
\begin{align*}
\sum_{i=0}^{r+1}  m_i 
=|(l_1\dots l_{2n-2})^{2^{n-2}}|+|c_{2^{n-1}}\cdots c_{r+1}|& 
\leq 2^{n-2}(2n-2)+j(2^{n+1}+2n-2)
\\
&\leq (2^{n-2}+j)(2n-2)+ j2^{n+1}.
\end{align*}
Since $r\leq 2^n-3$, we get $j\leq (r+1)/2$.
Using further that $2^{n-2}\leq (r+1)/2$, we obtain
$$
\sum_{i=0}^{r+1}  m_i \leq (r+1)(2n-2)+(r+1)2^{n} \leq (r+1)(2n+2r).
$$ 
This proves (\ref{item:lem:bounds-degree:Ms}). 
Item  (\ref{item:lem:bounds-degree:Ms'}) follows via the same argument, which concludes the lemma.  
\end{proof}

\begin{remark}
At least for small values of $n$, one can work out the integers $m_i,m_i',\tilde m_i$ and $\tilde m_i'$ from Definitions \ref{def:cici'etc} and \ref{def:citilde} explicitly. 
For instance, for $n=2$, we have $\langle c_0,\dots ,c_3\rangle=\langle 1,h_1,g_1g_2,h_1g_1g_2\rangle$ with $|h_1|=2$ and $|g_j|=4$. 
We thus obtain
$$
(m_0,m_1,m_2,m_3)=(0,2,8,10),\ \ (m'_0,m'_1,m'_2,m'_3)=(1,1,9,9) 
$$
$$ 
(\tilde m_0,\tilde m_1,\tilde m_2,\tilde m_3)=(5,5,5,5),\ \ (\tilde m'_0,\tilde m'_1,\tilde m'_2,\tilde m'_3)=(4,6,4,6)  .
$$
\end{remark}

\section{Proof of the main results}

\subsection{Quadric bundles with nontrivial unramified cohomology} 
The following theorem implies Theorem \ref{thm:H_nr} stated in the introduction.

\begin{theorem}\label{thm:H_nr:2}
Let $n$ and $r$ be positive integers with $r\leq 2^n-2$, and let $m$ be the unique integer with $2^{m-1}-1\leq r\leq 2^m-2$. 
Then there is a unirational complex $r$-fold quadric bundle $Y\longrightarrow \CP^n_\C$ with $H^m_{nr}(\C(Y)\slash \C,\mu_2^{\otimes m})\neq 0$.  
\end{theorem}

\begin{proof}
Since $ r+2\leq 2^m$, we may consider the homogeneous polynomials $c_i\in \C[x_0,\dots ,x_m]$ for $i=0,\dots ,r+1$ from Definition \ref{def:cici'etc}. 
Since $c_0=1$, the quadratic form $q=\langle c_0,\dots ,c_{r+1}\rangle$ defines an $r$-fold quadric bundle $Y'\longrightarrow \CP^m_\C$, whose generic fibre is of CTO type, see Corollary \ref{cor:CTO-bundles:1}.
Since $m\leq n$, the quadratic form $q$ defines also an $r$-fold quadric bundle $Y\longrightarrow \CP^n_\C$ which is stably birational to $Y'$.
Since unramified cohomology is a stable birational invariant,
$$
H^m_{nr}(\C(Y)\slash \C,\mu_2^{\otimes m})\cong H^m_{nr}(\C(Y')\slash \C,\mu_2^{\otimes m}) \neq 0 ,
$$
by Proposition \ref{prop:CTO}. 
Finally, $Y$ and $Y'$ are unirational by  Lemma \ref{lem:unirational}, because $c_0=1$ and $c_1=l_1l_2$, where $l_1,l_2\in \C[x_0,x_1,x_2]$ are general linear homogeneous polynomials, see (\ref{eq:li}) and (\ref{eq:phi}).
This proves Theorem \ref{thm:H_nr:2}.
\end{proof}

\begin{remark}
It follows from \cite[Theorem 3.1]{asok} that the quadric bundle $Y\longrightarrow \CP^n_\C$ from Theorem \ref{thm:H_nr:2} satisfies $H_{nr}^i(\C(Y)\slash \C,\mu_2^{\otimes i})=0$ for all $1\leq i\leq m-1$.  
\end{remark}

\subsection{Specialization theorems without resolutions}
 
Recall from Section \ref{subsec:conventions:specialize} what it means that a variety specializes to another variety.
The following specialization theorem is a generalization of Theorem \ref{thm:CTO} stated in the introduction. 

\begin{theorem}\label{thm:CTO:2} 
Let $X$ be a proper variety which specializes to a complex projective variety $Y$.
Suppose that there is a dominant rational map $f:Y\dashrightarrow \CP^n_\C$ with the following properties:
\begin{enumerate} 
\item some Zariski open and dense subset $U\subset Y$ admits a universally $\CH_0$-trivial resolution of singularities $\widetilde U\longrightarrow U$ such that the induced rational map $\widetilde U\dashrightarrow \CP^n_\C$ is a morphism whose generic fibre is proper over $K=\C(\CP^n)$. \label{item:CTO:2:1} 
\item the generic fibre $Y_{\eta}$ of $f$ is stably birational to a quadric of CTO type over $\C(\CP^n)$. 
\label{item:CTO:2:2} 
\end{enumerate}  
Then, no resolution of singularities of $X$ admits an integral decomposition of the diagonal.
In particular, $X$ is not stably rational.
\end{theorem}

\begin{proof}
To begin with, note that it suffices to prove the theorem after any extension of the base field of $X$.
Since $X$ specializes to a complex variety, we may thus assume that $X$ is defined over an algebraically closed field of characteristic zero.

Taking a suitable blow-up of some projective closure of $\widetilde U$, we obtain a proper birational morphism $\tau:\widetilde Y\longrightarrow Y$ with $\tau^{-1}(U)=\widetilde U$. 
By assumption (\ref{item:CTO:2:1}), $\tau^{-1}(U) \longrightarrow U$ is a universally $\CH_0$-trivial resolution of $U$.
Replacing $\widetilde Y$ by a log resolution which does not change $\widetilde U$, and which turns the complement $E:=\widetilde Y\setminus \widetilde U$ into a simple normal crossing divisor, we may additionally assume that $\tau$ is a resolution of singularities of $Y$ and each irreducible component $E_i$ of $E$ is smooth.  

By item (\ref{item:CTO:2:1}), $\widetilde U\dashrightarrow \CP^n_\C$ is a morphism which becomes proper when base changed to some open dense subset of $\CP^n_\C$.
Therefore, no component $E_i$ of $E$ dominates $\CP^n_\C$.

By item (\ref{item:CTO:2:2}), $Y_\eta$ is stably birational to a quadric $Q$ of CTO type over  $K=\C(\CP^n)$.
By Definition \ref{def:CTOtype} and 
Proposition \ref{prop:CTO}, there are nonzero elements $a_i,b_j\in K^\ast$ such that $\alpha_1:=(a_1,\dots ,a_{n-1},b_1)\in H^n(K,\mu_2^{\otimes n})$ pulls back to a nontrivial unramified class in $H^n_{nr}(K(Q)\slash \C,\mu_2^{\otimes n})$.
Since unramified cohomology is a stable birational invariant,
$$
0\neq \alpha':=f^\ast \alpha_1 \in H^n_{nr}(K(Y_\eta)\slash \C,\mu_2^{\otimes n}) .
$$
Applying Proposition \ref{prop:alpha':CTO} to the dominant rational map $\widetilde Y\dashrightarrow \CP^n_\C$, we see that $\alpha'|_{E_i}=0$ for any irreducible component $E_i$ of $E$.
Therefore, the assumptions of Proposition \ref{prop:degeneration} are satisfied and so no resolution of singularities of $X$ admits an integral decomposition of the diagonal. 
Since resolutions of singularities exist in characteristic zero, and because stably rational varieties admit integral decompositions of the diagonal (see Section \ref{subsec:CH0}), it follows that $X$ is not stably rational. 
This concludes Theorem \ref{thm:CTO:2}.
\end{proof}

\begin{proof}[Proof of Theorem \ref{thm:CTO}]
Let $X$ be a projective (or proper) variety which specializes to a complex projective variety $Y$ with a morphism $f:Y\longrightarrow S$ to a rational $n$-fold $S$ over $\C$, whose generic fibre is smooth and stably birational to a CTO type quadric $Q$ over $\C(S)$.
We may then consider the smooth locus $U:= Y^{\sm}$ of $Y$ and apply Theorem \ref{thm:CTO:2} to the universally $\CH_0$-trivial resolution $U\longrightarrow U$, given by the identity.
As the generic fibre of $f$ is smooth, the generic fibre of $U\longrightarrow \CP^n_\C$ coincides with $Y_\eta$ and so it is proper.
This shows that Theorem \ref{thm:CTO:2} applies and so $X$ is not stably rational.
This proves Theorem \ref{thm:CTO}.
\end{proof}

Theorem \ref{thm:CTO} has the following consequence.

\begin{corollary} \label{cor:CTO:2}
Let $n$ and $r$ be positive integers with $2^{n-1}-1\leq r \leq 2^n-2$. 
Let $e_0,\dots, e_{r+1}\in \C[x_0,\dots ,x_n]$ be nonzero homogeneous polynomials without common factor, whose degrees $d_i:=|e_i|$ are all odd or all even. 
Suppose that after setting $x_0=1$ and possibly multiplying each entry by some nonzero square, the quadratic form $\langle e_0,\dots ,e_{r+1} \rangle$ becomes similar to one of the quadratic forms $\langle c_0,\dots ,c_{r+1} \rangle$,  $\langle c'_0,\dots ,c'_{r+1} \rangle$, $\langle \tilde c_0,\dots , \tilde c_{r+1} \rangle$ or $\langle \tilde c'_0,\dots , \tilde c'_{r+1} \rangle$ from Corollary \ref{cor:CTO-bundles:1}. 

Then any projective variety 
 which specializes to the complex hypersurface 
$Y\subset \CP(\mathcal E)$ given by $\sum_ie_iz_i^2=0$, where 
$
\mathcal E=\bigoplus_{i=0}^{r+1}\mathcal O_{\CP^n_\C}(-\lfloor d_i\slash 2\rfloor)
$, 
is not stably rational. 
\end{corollary}

\begin{proof} 
Our assumption on the $e_i$ guarantees that $Y$ is integral, but note that $Y$ is not necessarily flat over $\CP^n_\C$, cf.\ Section \ref{subsec:quadricbundles}.
Nonetheless,  Corollary \ref{cor:CTO-bundles:1} implies that the generic fibre $Y_\eta$ of $Y\longrightarrow \CP^n_\C$ is a quadric of CTO type.  
This fact (or the assumption that $e_i\neq 0$ for all $i$) ensures that $Y_\eta$ is smooth. 
The corollary follows therefore from Theorem \ref{thm:CTO}.
\end{proof}


\subsection{Proof of Theorem \ref{thm:type} and some applications}


\begin{theorem}\label{thm:higherquadrics:2}
Let $n$ and $r$ be positive integers with $2^{n-1}-1\leq r \leq  2^n-2$, and let  $(d_i)_{0\leq i\leq r+1}$ be a tuple of non-negative integers of the same parity.
Consider  the non-negative integers $m_i$, $m_i'$, $\tilde m_i$ and $\tilde m_i'$ from Definitions \ref{def:cici'etc} and \ref{def:citilde}. 
Suppose that one of the following holds:
\begin{enumerate}
\item $d_0$ is even and $d_i\geq m_i $ for all $i$; \label{item:thm:higherquadrics2:mi}
\item  $d_0$ is odd and $d_i\geq m'_i $ for all $i$; 
\label{item:thm:higherquadrics2:mi'}
\item $d_0$ has the same parity as $\tilde m_0$ and  $d_i\geq \tilde m_i $ for all $i$; 
\item   $d_0$ has the same parity as $\tilde m_0'$  and $d_i\geq \tilde m'_i $ for all $i$. 
\end{enumerate}
Then a very general complex $r$-fold quadric bundle of type $(d_i)_{0\leq i\leq r+1}$ over $\CP^n_\C$ (see Definition \ref{def:type}) is not stably rational.
\end{theorem}

\begin{proof}
Choose a general linear homogeneous polynomial $l\in \C[x_0,\dots ,x_n]$, and let $c_i,c'_i,\tilde c_i$ and $\tilde c_i'$ be as in Definition \ref{def:cici'etc} .

If $d_0$ is even and $d_i\geq m_i $ for all $i$, then consider the homogeneous polynomials 
$$
e_0:=l^{d_0-m_0} \cdot c_0 \ \ \text{and}\ \ e_i:= x_0^{d_i-m_i} \cdot c_i  \ \ \text{for $i=1,\dots ,r+1$}.
$$
Since $l$ is general and the $d_i$ and $m_i$ are even, Lemma \ref{lem:specialize} and Corollary \ref{cor:CTO:2} show that a very general quadric bundle $X$ over $\CP^n_\C$ of type $(d_i)_{0\leq i\leq r+1}$ is not stably rational.

If $d_0$ is odd and  $d_i\geq m'_i $ for all $i$, then replace $c_i$ and $m_i$ by $c_i'$ and $m_i'$, respectively.
Since $m_i'$ is odd for all $i$, we may then argue as before.
If $d_0$ has the same parity as $\tilde m_i$ and $d_i\geq \tilde m_i$ for all $i$, then replace  $c_i$ and $m_i$ by $\tilde c_i$ and $\tilde m_i$, respectively, and argue as before.
If $d_0$ has the same parity as $\tilde m'_i$ and $d_i\geq \tilde m'_i$ for all $i$, then replace  $c_i$ and $m_i$ by $\tilde c'_i$ and $\tilde m'_i$, respectively, and argue as before. 
This finishes the proof of the theorem.
\end{proof}

\begin{proof}[Proof of Theorem \ref{thm:type}]
By Lemma \ref{lem:bounds-degree}, $\tilde m_i=2^n+n-1$ and $\tilde m'_i \leq 2^n+n$ for all $i$, and so Theorem \ref{thm:type} follows from Theorem \ref{thm:higherquadrics:2}.
\end{proof}

\begin{remark} \label{rem:defotype}
By Lemma \ref{lem:q=0:quadricbundle}, all examples in Theorems \ref{thm:type} and \ref{thm:higherquadrics:2} are smooth.
If $r\geq 2$, then all examples 
have rational deformations by Corollary \ref{cor:rational-defotype}.
If $(d_0,d_1)=(m_0,m_1)$ or $(d_0,d_1)=(m_0',m_1')$, 
then the examples in Theorem \ref{thm:higherquadrics:2} are unirational by Lemmas \ref{lem:unirational} and \ref{lem:bounds-degree}.
Unirationality of the examples in Theorem \ref{thm:type} is unknown.
\end{remark}

 \begin{proof}[Proof of Corollary \ref{cor:hyper}]
 The corollary follows from Lemma \ref{lem:singhyper} and Theorem \ref{thm:type}. 
\end{proof}

\begin{corollary} \label{cor:PlxPr}
Let $n$ and $r$ be positive integers with $2^{n-1}-1\leq r\leq 2^n-2$.  
Then a very general complex hypersurface $X\subset \CP^n_\C \times \CP^{r+1}_\C$ of bidegree $(d,2)$ with $d\geq 2^n+n-1$ is not stably rational.
\end{corollary}

\begin{proof}
A very general hypersurface of bidegree $(d,2)$ in $\CP^n_\C\times \CP^{r+1}_\C$ is nothing but a very general $r$-fold quadric bundle of type $(d,\dots ,d)$ over $\CP^n_\C$, because in the latter case $\mathcal E=\mathcal O_{\CP^n_\C}(-\lfloor d/2\rfloor)^{\oplus r+2}$ and so $\CP(\mathcal E)\cong \CP^n_\C\times \CP^{r+1}_\C$.
The corollary follows therefore from Theorem \ref{thm:type}. 
\end{proof}

\begin{corollary} \label{cor:doublecover}
Let $n,r$ be integers with $n\geq 2$ and $2^{n-1}-1\leq r \leq 2^n-2$ and put $N:=r+n$. 
Then a double cover of $\CP^{N}_\C$, branched along a very general complex hypersurface $Y\subset \CP^{N}_\C$ of even degree $d\geq 2^{n+1}+2n-2$ and with multiplicity $d-2$ along an $(r-1)$-plane is not stably rational.
\end{corollary}

\begin{proof}
By Lemma \ref{lem:doublePn}, we need to prove that a very general $r$-fold quadric bundle of type $(0,d-2,\dots ,d-2,d)$ is not stably rational if $d\geq 2^{n+1}+2n-2$ is even.
This follows from item (\ref{item:thm:higherquadrics2:mi}) in Theorem \ref{thm:higherquadrics:2}, because $m_0=0$ and $m_i\leq |l_1l_2\dots l_{2n-2}g_1g_2| = 2^{n+1}+2n-2$ for all $i$. 
\end{proof}

\subsection{Proof of Theorem \ref{thm:classification:2}} 

\begin{proof}[Proof of Theorem \ref{thm:classification:2}]
By Theorem \ref{thm:higherquadrics:2} and Remark \ref{rem:defotype}, there are many smooth unirational complex $r$-fold quadric bundles $Y\longrightarrow \CP^{m}_\C$ which are not stably rational. 
The product $X:=Y\times \CP^{n-m}_\C$ is then a smooth unirational complex $r$-fold quadric bundle over $S=\CP^{n-m}_\C\times \CP^m_\C$ which is not stably rational. 
This proves the theorem.
\end{proof}

\begin{remark}\label{rem:classification}
In the proof of Theorem 
\ref{thm:classification:2}, it is essential that Theorem \ref{thm:higherquadrics:2} yields smooth $r$-fold quadric bundles over rational bases which are not stably rational, non-rationality would not be enough.
\end{remark}

\begin{remark}\label{rem:asok} 
The cases $r=1,2$ in  Theorem 
 \ref{thm:classification:2}
 follows from \cite{voisin,HKT} and \cite{HPT}, respectively. 
If one allows singular bundles, the result follows from \cite{artin-mumford,CTO} if $r\leq 6$.
If $r\geq 7$, then the result is new even without the smoothness assumption. 
\end{remark}

\subsection{Proof of Theorem \ref{thm:defo}}  

\begin{theorem}\label{thm:higherquadric} 
Let $n,r$ and $d$ be integers, with $d$ even if $r$ is even, and such that $n\geq 2$,   
$2^{n-1}-1\leq r \leq 2^{n}-2$  
and $d\geq 2(n+r)(r+1)$. 

There is a smooth complex projective family  $\pi:\mathcal X\longrightarrow B$ over a complex variety $B$, such that each fibre $X_b=\pi^{-1}(b)$ is a smooth $r$-fold quadric bundle over $\CP^n_\C$, degenerated over a hypersurface of degree $d$ in $\CP^n_\C$, satisfying the following:  
\begin{enumerate} 
\item for very general $t\in B$, the $r$-fold quadric bundle $X_t$ over $\CP^n_\C$ is not stably rational;\label{item:nonrational}
\item all fibres of $\pi$ are unirational and, if $r\geq 2$, then some fibres are rational;\label{item:unirational+rational}
\item  if $r\geq 3$ and $d$ is even, the set $\{ b\in B\mid X_b\ \text{is rational}\}$ is dense in $B$ for the analytic topology.\label{item:density}
\end{enumerate} 
\end{theorem}

\begin{proof} 
We first define some non-negative integers $d_i$ for $i=0,\dots ,r+1$ of the same parity 
and use the notation from Definition \ref{def:cici'etc}.
If $d$ is even, we put
$d_i:=m_i$ for $i=0,\dots ,r$ and $d_{r+1}:=d-\sum_{i=0}^r m_i$.
If $d$ is odd, we define
$d_i:=m'_i$ for $i=0,\dots ,r$ and $d_{r+1}:=d-\sum_{i=0}^r m'_i$.
Since $d\geq 2(n+r)(r+1)$, Lemma \ref{lem:bounds-degree} ensures $d_{r+1}\geq m_{r+1}$ and $(d_0,d_1)=(0,2)$ if $d$ is even, and $d_{r+1}\geq m_{r+1}'$ and $(d_0,d_1)=(1,1)$ if $d$ is odd. 

Let $
\mathcal E^\vee:=\bigoplus_{i=0}^{r+1} \mathcal O_{\CP^n_\C}(\lfloor d_i/2\rfloor) 
$
and consider the complex vector space $V':=H^0(\CP^n_\C,\Sym^2(\mathcal E^\vee)\otimes \mathcal O_{\CP^n_\C}(d_0))$. 
We identify points in $V'$ with symmetric matrices $A=(a_{ij})_{0\leq i,j\leq r+1}$.
To such a matrix, we associate the minor $M(A):=(a_{ij})_{i,j\in \{0,1,2,n\}}$, which is a symmetric $4\times 4$ matrix.
We define $V\subset V'$ as the linear subspace given by all symmetric matrices $A=(a_{ij})$ with $a_{ij}\in \C[x_0,x_1,x_2]$ for all $i,j\in \{0,1,2,n\}$, and such that $a_{i0}=0$ for $i=1,2,n$ if $d$ is even and $n\geq 3$.  
We let $B\subset \CP(V)$ be the subset of points $[A]\in \CP(V)$ such that $A$ defines a smooth $r$-fold quadric bundle of type $(d_i)_{0\leq i\leq r+1}$ over $\CP^n_\C$; if $n\geq 3$, then we also assume that $M(A)$ defines a smooth quadric surface bundle of type $(d_0,d_1,d_2,d_n)$ over $\CP^2_\C$. 
By Bertini's theorem, $B$ is an open dense subset of $\CP(V)$. 
There is a universal hypersurface $\mathcal X\subset B\times \CP(\mathcal E)$.
Projection to the first factor gives a smooth projective morphism $\pi:\mathcal X\longrightarrow B$ of complex varieties.
The fibre $X_b$ above $b\in B$ is a smooth $r$-fold quadric bundle over $\CP^n_\C$, which degenerates over a hypersurface of degree $d$ in $\CP^n_\C$. 
Let $t\in B$ be very general.
Since $c_0,c_1,c_2,c_n,c'_0,c'_1,c_2',c_n'\in \C[x_0,x_1,x_2]$ by (\ref{eq:li}) and (\ref{eq:phi}), $X_t$ specializes by Lemma \ref{lem:specialize} to the hypersurface $Y\subset \CP(\mathcal E)$, given by $\sum_{i=0}^{r+1}e_iz_i^2=0$, where $e_i=c_i$ (resp.\ $e_i=c_i'$) for $i=0,\dots ,r$ and $e_{r+1}=x_0^{d_{r+1}-m_{r+1}}c_{r+1}$ (resp.\ $e_{r+1}=x_0^{d_{r+1}-m_{r+1}'}c_{r+1}'$), if $d$ is even (resp.\ odd), and 
where we use the notation from Definition \ref{def:cici'etc}.
It thus follows from Corollary \ref{cor:CTO:2} that $X_t$ is not stably rational.
This proves item (\ref{item:nonrational}).

Recall $(d_0,d_1)\in \{(1,1),(0,2)\}$. 
Up to replacing $B$ by some open dense subset which is given by a certain genericity assumption on $(a_{ij})_{0\leq i,j\leq 1}$, Lemma \ref{lem:unirational} thus ensures that all fibres of $\pi$ are unirational.
If $r\geq 2$, then our assumptions imply $r\geq n$.
As in Lemma \ref{lem:rationaldefotype}, Bertini's theorem shows then that we may additionally assume that $B$ contains points which correspond to matrices $A=(a_{ij})$ with $a_{r+1,r+1}=0$. 
The corresponding quadric bundles admit sections and so they are rational.
This proves item (\ref{item:unirational+rational}).

Let us now assume that $r\geq 3$ and $d$ is even. 
Then, $n\geq 3$ and $(d_0,d_1,d_2,d_n)=(0,2,2,4)$ (see (\ref{eq:phi})) and we consider the  vector space $W$ of symmetric $4\times 4$ matrices from Proposition \ref{prop:voisin}.
There is a dominant morphism 
$$
M:B\longrightarrow \CP(W),\ \ [A]\mapsto [M(A)].
$$ 
If the complex quadric surface bundle over $\CP^2_\C$ which is defined by $M(A)$ admits a rational section, then the complex $r$-fold quadric bundle over $\CP^n_\C$ defined by $A$ admits a rational section as well.
By Proposition \ref{prop:voisin}, the set of points $[M(A)]\in \CP(W)$ with that property is dense for the analytic topology.
This proves item (\ref{item:density}), i.e.\ $\{b\in B\mid \text{$X_b$ is rational}\}$ is dense in $B$ for the analytic topology.
This concludes Theorem \ref{thm:higherquadric}. 
\end{proof}

\begin{proof}[Proof of Theorem \ref{thm:defo}]
The case $r=1,2$ follows from \cite{HPT}, because the examples treated there (general hypersurfaces of bidegree $(2,2)$ in $\CP^2_\C\times \CP^3_\C$) are both, quadric surface bundles over $\CP^2_\C$, as well as conic bundles over $\CP^3_\C$, see also \cite[Remark 10]{HPT}.
The case $r\geq 3$ follows from Theorem \ref{thm:higherquadric}.
\end{proof}

\begin{remark}\label{rem:parity:d}
The restriction on the parity of $d$ if $r$ is even is necessary in Theorem \ref{thm:higherquadric}.
Indeed, the Fano variety of $m$-planes on a smooth quadric of dimension $2m$ has two connected components, and so any smooth $2m$-fold quadric bundle over $\CP^n_\C$ gives rise to a double cover of $\CP^n_\C$, branched along the degeneration divisor.
This forces the degree of the degeneration divisor to be even.
\end{remark}

\begin{remark}\label{rem:density}
It is conceivable that item (\ref{item:density}) in Theorem \ref{thm:higherquadric} holds for $r\geq 2$ and without the restriction on the parity of $d$.
To prove this, it would be enough to generalize Proposition \ref{prop:voisin} to other quadric surface bundles of type $(d_0,\dots ,d_3)$ over $\CP^2_\C$. 
Even though this problem seems tractable with the existing methods, we do not try to pursue this here.
\end{remark}

\section*{Acknowledgements}  
I am very grateful to J.-L.\ Colliot-Th\'el\`ene and to the excellent referee for carefully reading this paper and for many comments which significantly improved the exposition.
Thanks to B.\ Totaro for a reference and to A.\ Asok for useful comments.
I had very useful conversations about topics related to this paper with several people, including R.\ Beheshti, O.\ Benoist and L.\ Tasin.
The results of this article where conceived while the author was member of the SFB/TR 45 ``Periods, Moduli Spaces and Arithmetic of Algebraic Varieties''.


\end{document}